\documentclass[reqno]{amsart}
\usepackage[T1]{fontenc}
\usepackage[latin1]{inputenc}
\usepackage[english]{babel}
\usepackage[babel]{csquotes}

\usepackage{cite}
\usepackage{amssymb}
\usepackage{amsmath}
\usepackage{amsthm}
\usepackage{latexsym}
\usepackage{graphicx}
\usepackage{mathrsfs}
\usepackage{bbm}
\usepackage{verbatim}
\usepackage[colorlinks=true, pdfstartview=FitV, linkcolor=blue, citecolor=blue, urlcolor=blue]{hyperref}

\usepackage[usenames,dvipsnames]{color}

\newtheorem{thm}{Theorem}[section]

\newtheorem{lem}[thm]{Lemma}

\newtheorem{remark}[thm]{Remark}

\numberwithin{equation}{section}

\def\enne{\mathbb{N}}
\def\zeta{\mathbb{Z}}

\def\erre{\mathbb{R}}
\def\R{\mathbb{R}}

\def\P{\mathbb{P}}

\def\H{\mathcal H}
\def\E{\mathop{{}\mathbb{E}}}
\def\cH{\mathcal{H}}
\def\cL{\mathscr{L}}
\def\cF{\mathscr{F}}
\def\cB{\mathscr{B}}
\def\eps{\varepsilon}
\def\cP{\mathscr{P}}

\def\OO{\mathcal{O}}
\renewcommand{\d}{{\mathrm d}}
\renewcommand{\div}{\operatorname{div}}

\def\beq{\begin{equation}}
\def\eeq{\end{equation}}

\def\to{\rightarrow}
\def\wto{\rightharpoonup}
\def\wstarto{\stackrel{*}{\rightharpoonup}}
\def\embed{\hookrightarrow}

\def\norm #1{\left\|#1\right\|}

\def\sp #1#2{\left<#1,#2\right>}
\newcommand\ip\sp

\newenvironment{system} 
{\left\lbrace\begin{array}{@{}l@{}}}
{\end{array}\right.}

\begin{document}
\title[Separation for the stochastic Allen-Cahn equation]
{Random separation property\\ for
stochastic Allen-Cahn-type equations}

\author{Federico Bertacco}
\address[Federico Bertacco]{Department of Mathematics, Imperial College London, 
London SW7 2AZ, United Kingdom}
\email{f.bertacco20@imperial.ac.uk}

\author{Carlo Orrieri}
\address[Carlo Orrieri]{Department of Mathematics, 
Universit\`a di Pavia, Via Ferrata 1, 27100 Pavia, Italy}
\email{carlo.orrieri@unipv.it}
\urladdr{http://www-dimat.unipv.it/orrieri}

\author{Luca Scarpa}
\address[Luca Scarpa]{Department of Mathematics, Politecnico di Milano, 
Via E.~Bonardi 9, 20133 Milano, Italy.}
\email{luca.scarpa@polimi.it}
\urladdr{http://www.mat.univie.ac.at/$\sim$scarpa}



\begin{abstract}
We study a large class of stochastic $p$-Laplace Allen-Cahn equations
with singular potential.
Under suitable assumptions on the (multiplicative-type) noise we first prove existence, uniqueness, and regularity of variational solutions.
Then, we show that a random separation property holds, i.e.~almost 
every trajectory is strictly separated in space and time from the potential barriers.
The threshold of separation is random, and 
we further provide
exponential estimates on the probability of separation from the barriers.
Eventually, we exhibit a convergence-in-probability 
result for the random separation threshold 
towards the deterministic one, as the noise vanishes, and 
we obtain an estimate of the convergence rate. 
\end{abstract}

\maketitle
\setcounter{tocdepth}{1}
\tableofcontents


\section{Introduction}
\setcounter{equation}{0}
\label{sec:intro}

In this paper we study a class of stochastically perturbed Allen-Cahn-type equations with a particular emphasis on the separation property of their solutions from the potential barriers. To motivate our interest let us firstly spend some words on the (by now classical)
deterministic models and on some of the problems arising in their stochastic counterpart. 

\textit{Deterministic setting.} Allen-Cahn equations are particular instances of the broad class of phase field models and are suitable to describe i.e.~the evolution of the normalized density $u$ of (one of) the phases involved in a phase separation process.  
The phase-field variable $u$ is supposed to 
take values in $[-1,1]$, with $\{u=1\}$ and $\{u=-1\}$ representing the two 
so-called pure-regions and $\{-1<u<1\}$ standing for the narrow 
diffuse interface between them.
Given a bounded domain $\OO \subset \R^d$ and a time horizon $T >0$, 
the classical (simplified) formulation of the Allen-Cahn equation is 
$$\partial_t u - \Delta u + u^3 - u = 0, \qquad \text{ in } (0,T) \times \OO\,,$$
under suitable boundary and initial conditions.
This can be derived in a variational fashion as the $L^2$-gradient flow of the Ginzburg-Landau free energy
\begin{equation*}
\mathcal{E}_{2,\operatorname{pol}}(u):=\int_\OO 
\left(\frac{1}{2}|\nabla u|^2 + \frac{1}{4}(u^2-1)^2\right)\,.
\end{equation*} 
It is evident from the form of the free-energy functional that 
the Allen-Cahn dynamics result from the interplay of two factors:
on the one hand, the tendency of each phase to concentrate 
at the pure phases (the two global minima of the potential),
and on the other hand the penalization of the
space oscillations of the phase-field variable.

In the last decades, numerous 
generalizations of Allen-Cahn models have been proposed in literature. 
At a formal level, free energies and corresponding gradient flows can
be written in the very broad form
\begin{equation*}
\mathcal{E}_{\varphi,\operatorname{F}}(u):= 
\int_\OO\left( \varphi(\nabla u) + F(u)\right)\,, \qquad \partial_t u  + \partial \Phi(u) + F'(u) = 0\,, \quad \text{ with } \Phi(u) := \int_\OO \varphi(\nabla u)
\end{equation*} 
where the choice of the gradient and potential terms depends on the particular physical model under consideration. As far as the choice of $F$ is concerned, 
notice that the polynomial potential in $\mathcal E_{2,\operatorname{pol}}$,
despite being relatively easy to handle from the mathematical point of view,
is totally ineffective in the modelling construction. Indeed, 
minimizers $u$ of the polynomial energy $\mathcal E_{2,\operatorname{pol}}$
do not satisfy the physically relevant constraint $u\in[-1,1]$
that one would expect from the very definition of the relative concentration $u$.
In this direction, 
the most relevant choices of the double-well potential $F$ are 
the so-called singular potentials instead, i.e.~defined on $[-1,1]$ only.
The typical double-well potential 
which is classically proposed in thermodynamics is the 
the so-called \textit{logarithmic potential} 
$$F_{\log}(r):=\frac{\theta}{2}((1+r)\ln(1+r)+(1-r)\ln(1-r))-\frac{\theta_0}{2}r^2,$$ 
with $r \in [-1,1]$ and $0 < \theta < \theta_0$ being given constants. 
In this case, the two minima are strictly contained inside the physical relevant domain and the derivative of $F_{\log}$ blows up at $\pm 1$, thus forcing the order parameter $u$ to take values in $[-1,1]$. 
For what concerns the gradient part, a natural candidate is the  generalization of the Dirichlet energy  given by
$$\Phi_p(u) = \frac{1}{p}\int_\OO|\nabla u|^p \,, \qquad p \in [1,+\infty).$$ 
At the level of the flow, the above energy produces a  (nonlinear) $p$-Laplace operator $\Delta_p u = \div \left(|\nabla u|^{p-2} \nabla u \right)$, which reduces to the classical linear diffusion with the choice $p =2$.

Under the physically relevant choice of the logarithmic potential $F_{\log}$, 
the solution $u$ of the corresponding deterministic gradient flow
takes values in the natural range $[-1,1]$. However, 
it is possible to say something more. A qualitative analysis of the equation 
shows that under $F_{\log}$ solutions to the Allen-Cahn equation satisfy 
a separation principle, meaning that is the initial concentration $u_0$
is strictly separated from the pure phase $\pm1$ (i.e.~one starts with an 
actual mixture) then the solution $u$ remains strictly separated from $\pm1$
at all times everywhere in $\OO$. More precisely, this can be formulated as
\[
  \sup_{x\in\overline\OO}|u_0(x)| < 1 \qquad\Rightarrow\qquad
  \sup_{(t,x)\in[0,T]\times\overline\OO}|u(t,x)| < 1\,.
\]
The strict inequality above is the essence of the separation property, as it
means that $u$ cannot accumulate towards the 
barriers $\pm1$ of the singular potential and uniquely defines 
a positive separation layer.
Not only is this extremely 
natural since $F_{\log}$ has minima inside the interval $(-1,1)$, 
but it also has important implications in the underlying thermodynamical 
derivation of the model.

The literature on deterministic Allen-Cahn, Cahn-Hilliard, 
and general phase field systems is extremely vast. Here we just mention a few contributions dealing with confining potential or nonlinear diffusions: 
\cite{cal-colli, colli-sprek-optACDBC, fol-stra, mir-sao-tal}.
Results specifically concerning the separation property from the barriers 
in the case of confining potential are studied in several contexts e.g.~in
\cite{bcst2, GGG, GGM, GGW, lon-pez, mir-sao-tal}.

\textit{Stochastic setting.}
The stochastic counterpart of Allen-Cahn-type equations reads as
\[  \d u +\partial \Phi (u)\,\d t + F'(u)\,\d t = \mathcal H(u)\,\d W, \]
where the operator $\mathcal H$ is introduced to suitably inject the random (generally Gaussian) perturbation $W$ into the physical domain. This class of stochastic equations has not been fully investigated yet. 
Even in the presence of a linear diffusion $\partial \Phi (u) = -\Delta u$, a general well-posedness theory for stochastic Allen-Cahn equation with confining potentials (e.g. $F(u) = F_{\log}(u))$ and additive noise is not yet available.
Let us only mention that significant results have been obtained for potentials with polynomial growth, see e.g. \cite{gess} and the references therein, or even without growth conditions as in \cite{mar-scar-diss}, but still defined on the whole real line;  see also \cite{liu, gess-tolle} where $p$-Laplace operators are considered, and 
\cite{orr-scar} for the case of dynamic boundary conditions.
\\
It is thus interesting to investigate whether it is possible, and under which assumptions on the noise, to restore well-posedness of the system in the presence of a confining potential.
In this direction, the only strategy that we are aware of is presented in \cite{bertacco} (see also \cite{bauz-bon-leb} for a simpler model with one-dimensional noise), where suitable condition on the multiplicative noise is presented so to ``compensate'' the singular character of the drift term. 
Morally, the idea is to switch off the noise as soon as the solution touch $\pm 1$ in order to confine it inside the physically relevant domain.
This permits to take advantage of the classical theory of Pardoux, Krylov and Rozovskii for a regularized version of the equation and to obtain uniform (in the approximation parameter) estimates to pass to the limit.
For a physical motivation/application of this class of noises we refer to \cite{orr-roc-scar} for what concerns tumour growth models and to \cite{dar-gess-gnann-grun, grun-metzger, scar-mobility} for stochastic thin-film and Cahn-Hilliard equations where a mobility term is introduced together with the singular potential.

\textit{Our contribution.} 
The qualitative study of the separation property for stochastic Allen-Cahn equations
is currently an open problem in literature, despite having 
important implications on the thermodynamical model beneath.
The aim of the this work is to provide a first answer in this direction, 
and the research is twofold: on the one hand we aim at
investigating the separation of the trajectories of the stochastic equation 
from the potential barriers, and on the other hand we want to 
evaluate the effect of the noise on the
random separation layer, compared with the deterministic one,
in terms of the noise intensity. 

We consider stochastic Allen-Cahn equations of the form
\begin{equation}\label{eq:intro}
\begin{aligned}
  \d u - \Delta_p u\,\d t + F'(u)\,\d t = \mathcal H(u)\,\d W \qquad&\text{in } (0,T)\times\OO\,,\\
  u=0 \qquad&\text{in } (0,T)\times\partial\OO\,,\\
  u(0)=u_0 \qquad&\text{in } \OO\,.
\end{aligned}
\end{equation}
where $\Delta_p u = \div \left(|\nabla u|^{p-2} \nabla u \right)$ is the p-Laplacian operator with $p \in [2,+\infty)$, $\OO \subset \R^d$ is a bounded domain with regular boundary and $W$ is a cylindrical Wiener process.
The map $F: (-1,1) \to [0,+\infty)$ is a confining potential whose derivative $F'$ blows up at the boundary of the physical relevant domain and the diffusion term $\H$ is chosen so to compensate its singular character.
In particular our setting fully covers the physically relevant case 
of the logarithmic potential $F_{\log}$ (see Section~\ref{sec:main} for 
details).

In this framework we are able to prove well-posedness of the problem \eqref{eq:intro} along with refined estimates when the initial datum is more regular, see Theorems \ref{thm1} and \ref{thm1+} for precise statements.
To prove existence of a solution, we adapt the strategy presented in \cite{bertacco} to the case of nonlinear $p$-Laplace operator. 
From a mathematical point of view, the introduction of $p$-Laplace diffusion with $p >d$
allows to gain for strong solution enough spatial regularity so that a pointwise 
space-time evaluation of the solutions is at hand:
\[
  \norm{u(\omega)}_{L^\infty(Q)}=\sup_{(t,x)\in[0,T]\times\overline\OO}|u(\omega,t,x)|
  \qquad\text{ for }\,\P\text{-a.e.~}\omega\in\Omega\,.
\]

Once well-posedness is settled, natural questions on properties of the solutions near the barriers of the confining potential arise. Which are the main differences with respect to the deterministic setting? How the stochastic separation layer depends on $\omega \in \Omega$? Is it possible to show a convergence result towards the deterministic one?
The purpose of the paper is indeed to address the above questions for	 solutions of \eqref{eq:intro}. The first result we show in this direction is a \textit{pathwise} separation property: when the initial datum is  strictly separated from the barriers, then almost every trajectory remains separated for all times.
Specifically, if  $\|u_0\|_{L^\infty(\OO)} = 1-\delta_0$ with $\delta_0 \in (0,1)$, then for almost every $\omega \in \Omega$, there exists $\delta(\omega)\in(0,1)$ such that 
\[
\sup_{(t,x)\in[0,T]\times\overline\OO}|u(\omega,t,x)|
 \leq 1-\delta(\omega)\,,
 \] 
see Theorem \ref{thm2} for a precise statement. 

In our second main result (see Theorem \ref{thm3}) we precisely quantify the probability of separation in an exponential fashion: there exist $L >0$, $\delta_* \in (0,\delta_0)$ and $\rho >0$ such that 
\[  \P\left(\sup_{(t,x)\in[0,T]\times\overline\OO}
  |u(t,x)|\geq1-\delta\right)\leq  \exp\left(-L\delta^{-\rho}\right) \qquad
  \forall\,\delta\in(0,\delta_*)\,, \]
To obtain the above estimate we combine \textit{boundary} and \textit{interior} $L^{\infty}$-estimates for solutions, which in turn are derived from Bernstein inequalities for suitable stochastic integrals.   
A caveat: to study separation properties we have to slightly strengthen the hypotheses on the diffusion coefficients (w.r.t.~the minimal ones required for well-posedness) assuming that also a certain number of derivatives vanish in $\pm 1$; see Assumption {\bf D} for details and Section \ref{sec:ex} for concrete examples.

We eventually introduce a parametrized family $(u_\varepsilon)_{\varepsilon \geq 0}$ of solutions to 
\[ \d u_\eps - \Delta_p u_\eps\,\d t + F'(u_\eps)\,\d t 
  = \sqrt\eps\mathcal H(u_\eps)\,\d W  \]
and we investigate the convergence in probability, as $\varepsilon \downarrow 0$, of the random separation layer, say $\Lambda_\varepsilon$, toward the deterministic one $\delta_0$.
Specifically, if we assume the initial data to be energetically well-prepared, the deterministic solution (with $\varepsilon = 0$) is separated from the barrier of at least the deterministic threshold $\delta_0 = 1 - \|u_0\|_{L^\infty(\OO)}$.  
Taking this into account, we are able to prove that
\[ \|u_\varepsilon\|_{L^\infty((0,T) \times \OO)} \overset{\P}{\longrightarrow} \|u_0\|_{L^\infty(\OO)}\]  
which in turn is equivalent to the convergence $\Lambda_\varepsilon \overset{\P}{\longrightarrow} \delta_0$.
An exponential upper bound on the velocity of the convergence is also given, we refer to Theorem \ref{thm4} for a precise formulation of the result.

\textit{Structure of the paper.} The paper is organized as follows. 
Section \ref{sec:main} contains the assumptions and the precise formulation of the main results.
In Section \ref{sec:proof1} we prove basic and refined well-posedness of the system via three crucial estimates.
The Separation property for almost all trajectories in then discussed in Section \ref{sec:proof2}. In Section \ref{sec:proof3} we carefully investigate the distribution of the separation layer and exhibit exponential estimates for the probability of separation.
Section \ref{sec:proof4} is devoted to the convergence of the separation layer as the noise vanishes.
Finally, in Section \ref{sec:ex} we provide some examples and application of the results obtained in the paper.


\section{Main results}
\label{sec:main}
In this section we state the precise assumptions on the setting and the data 
of the problem, and we present the main results of the work.

{\bf Setting}. First of all, let $(\Omega,\cF,(\cF_t)_{t\in[0,T]},\P)$ be a filtered probability space
satisfying the usual conditions, where $T>0$ is a fixed time. Let also $W$ be a 
cylindrical Wiener process on a separable Hilbert space $U$, and let us 
fix once and for all a complete orthonormal system $(e_j)_{j\in\enne}$ on $U$.
Secondly, let $\OO\subset\erre^d$ ($d\geq1$) be a bounded domain 
with Lipschitz boundary $\partial\OO$ and outward
normal unit vector $\bf n$.
We define the functional spaces 
\[
  H:=L^2(\OO)\,, \qquad V_p:=W^{1,p}_{0}(\OO)\,, \quad p\geq2\,,
\]
endowed with their usual norms $\norm{\cdot}_H$ 
and $\norm{\cdot}_{V_p}$, respectively. 
In the case $p = +\infty$ we define $W^{1,\infty}_0(\OO):= \lbrace u \in W^{1,\infty}(\OO) : \gamma(u) = 0 \,\text{ a.e. on } \partial \OO \rbrace$ with $\gamma: W^{1,\infty}(\OO) \to L^\infty(\OO)$ the trace operator.  
The Hilbert space $H$
is identified to its dual through the Riesz isomorphism, so that
we have the dense, continuous, and compact inclusions
\[
  V_p \embed H \embed V_p^*\,.
\] 
We define the p-Laplacian operator $-\Delta_p: V_p \to V_p^*$ as follows
\[ \ip{-\Delta_pu}{v}_{V_p^*,V_p}  := \int_\OO |\nabla u|^{p-2}\nabla u \cdot \nabla v\,, 
\qquad u,v \in V_p\,.\]
Notice that $-\Delta_p: V_p \to V_p^*$ is (bounded) monotone and demicontinuous. In particular, if $u_j \to u$ strongly in $V_p$ then $\Delta_pu_j \wto \Delta_pu$ weakly in $V_p^*$.
Moreover, for $p \in [2,+\infty)$ we introduce the functional 
\begin{equation*}
\Phi_{p}:= 
\begin{system}
\frac{1}{p}\int_{\OO} |\nabla u|^p\,, \qquad \,\text{ if } u \in V_p  \\
+ \infty\,, \qquad \quad \qquad \text{ if } u \in H \setminus V_p.  \\
\end{system}
\end{equation*}
It is well konwn that $\Phi_p$ is lower semi-continuous and convex with corresponding subdifferential $\partial \Phi_p$ a realization of the $p$-laplace operator on $\OO$ with Dirichlet boundary condition.  

For every $R>0$, we define the closed ball
\[
  \mathcal B_R:=\{\varphi\in H:\; |\varphi|\leq R \quad\text{a.e.~in } \OO\}\,.
\]
We use the classical notation $Q:=(0,T)\times\OO$ and $Q_t:=(0,t)\times\OO$
for all $t\in(0,T]$.

{\bf Notation}.
For every Banach spaces $E_1$ and $E_2$, the symbol $\cL(E_1,E_2)$
denotes the space of linear continuous operators from $E_1$ to $E_2$;
if $E_1$ and $E_2$ are also Hilbert spaces, the space of Hilbert-Schmidt 
operators from $E_1$ to $E_2$ is denoted by $\cL^2(E_1,E_2)$.
We denote by $\cP$ the progressive sigma algebra on $\Omega\times[0,T]$.
For every $s,r\in[1,+\infty]$ and for every Banach space $E$
we use the classical symbols $L^s(\Omega; E)$ and $L^r(0,T; E)$
to indicate the spaces of strongly measurable Bochner-integrable functions 
on $\Omega$ and $(0,T)$, respectively. 
Moreover, 
for all $s,r\in[1,+\infty)$ we use the special symbol 
$L^s_\cP(\Omega;L^r(0,T; E))$ to specify 
that measurability is intended with respect to $\cP$.
In the case that $s\in(1,+\infty)$,
$r=+\infty$, and $E$ is a separable and reflexive,
we explicitly set 
\[
  L^s_w(\Omega; L^\infty(0,T; E^*)):=
  \left\{v:\Omega\to L^\infty(0,T; E^*) \text{ weakly* meas.}\,:\,
  \E\norm{v}_{L^\infty(0,T; E^*)}^s<\infty
  \right\}\,,
\]
so that by 
\cite[Thm.~8.20.3]{edwards} we have the identification
\[
L^s_w(\Omega; L^\infty(0,T; E^*))=
\left(L^{s/(s-1)}(\Omega; L^1(0,T; E))\right)^*\,.
\]

{\bf Assumption A}. Let $p\in[2,+\infty)$ and set $q:=\frac{p}{p-1}\in(1,2]$. We introduce 
\[
  \hat\gamma_p:\erre^d\to[0,+\infty)\,,\qquad
  \hat\gamma_p(x):=\frac1p|x|^p\,, \quad x\in\erre^d\,.
\] 
Clearly, $\hat\gamma_p$ is convex and continuously differentiable, with 
differential given by
\[
  \gamma_p:\erre^d\to\erre^d\,, \qquad
  \gamma_p(x):=D\hat\gamma_p(x)=|x|^{p-2}x\,,\quad x\in\erre^d\,.
\]

{\bf Assumption B}. Let $F:(-1,1)\to[0,+\infty)$ be of class $C^2$, with $F'(0)=0$, such that 
\[
  \lim_{r\to(-1)^+}F'(r)=-\infty\,, \qquad
  \lim_{r\to1^-}F'(r)=+\infty\,,
\]
and
\[
  \exists\, C_F\geq0:\qquad F''(r)\geq - C_F \quad\forall\,r\in(-1,1)\,.
\]
This implies in particular that the operator
\[
  \beta:(-1,1)\to\erre\,, \qquad \beta(r):=F'(r) + C_Fr\,, \quad r\in(-1,1)\,,
\]
is maximal monotone as a graph in $\erre\times\erre$. Moreover, 
let us define $r_F, R_F\in(-1,1)$ as
\begin{align*}
  r_F&:=\sup\{r\in(-1,1):F'(z)\leq0\quad\forall\,z\in(-1,r)\}\,,\\
  R_F&:=\inf\{r\in(-1,1):F'(z)\geq0\quad\forall\,z\in(r,1)\}\,,
\end{align*}
so that $-1<r_F\leq R_F<1$ and $F'(r_F)=F'(R_F)=0$.

{\bf Assumption C}. Let 
\[
  (h_k)_{k\in\enne}\subset W_0^{1,\infty}(-1,1)\,,
\]
be such that 
\[
F''h_k^2\in L^\infty(-1,1) \qquad\forall\,k\in\enne
\]
and
\[
  C_{\mathcal H}^2:=\sum_{k=0}^\infty
  \left(\norm{h_k}_{W^{1,\infty}(-1,1)}^2
  +\norm{F''h_k^2}_{L^\infty(-1,1)}\right)
  <+\infty\,.
\]
Then, in particular it is well-defined the operator 
\[
  \H:\mathcal B_1\to \cL^2(U,H)\,,\qquad
  \H(v):e_k\mapsto h_k(v)\,, \quad v\in\mathcal B_1\,,\quad k\in\enne\,,
\]
which is $C_\H$-Lipschitz-continuous with respect to the 
metric of $H$ induced on $\mathcal B_1$.

We are now ready to state our main results.

\begin{thm}[Well-posedness]
  \label{thm1}
  Assume {\bf A--B--C}, and let 
  \[
  u_0\in H\,, \qquad F(u_0)\in L^1(\OO)\,.
  \]
  Then, there exists a unique $u$ with
  \begin{align*}
  &u\in L^p_\cP(\Omega; C^0([0,T]; H))\cap L^p_\cP(\Omega; L^p(0,T; V_p))\,, \\
  &\gamma_p(u) \in L^q_\cP(\Omega; L^q(0,T; L^q(\OO)^d))\,,\\
  &F'(u)\in L^2_\cP(\Omega; L^2(0,T; H))\,,
  \end{align*}
  such that for every $v \in V_p$
\begin{equation}\label{variational_sol}
\begin{aligned}
  \int_\OO u(t) v \, + &\int_0^t \int_\OO \gamma_p(\nabla u(s)) \cdot \nabla v\,\d s + \int_0^t \int_\OO  F'(u(s)) v \,\d s \\
  &= 
  \int_\OO u_0 v + \int_\OO \left(\int_0^t  \H(u(s))\,\d W(s) \right) v 
\end{aligned}
\end{equation}
for every $t \in [0,T]$, $\P$-a.s.
  In particular, it holds that 
  \[
  |u|\leq 1 \quad\text{a.e.~in } \Omega\times(0,T)\times\OO\,.
  \]
\end{thm}

\begin{thm}[Refined well-posedness]
  \label{thm1+}
  Assume {\bf A--B--C}, and let 
  \[
  u_0\in V_p\,, \qquad F(u_0)\in L^1(\OO)\,.
  \]
  Then, there exists a unique $u$ with
  \begin{align*}
  &u\in L^p_\cP(\Omega; C^0([0,T]; H))\cap L^p_w(\Omega; L^\infty(0,T; V_p))\,, \\
  &\Delta_p u, F'(u) \in L^2_\cP(\Omega; L^2(0,T; H))\,,
  \end{align*}
  such that 
  \begin{equation}\label{var_sol+}
  u(t) - \int_0^t\Delta_pu(s)\,\d s + \int_0^tF'(u(s))\,\d s = 
  u_0 + \int_0^t\H(u(s))\,\d W(s) 
  \end{equation}
 for every $t \in [0,T]$, $\P$-a.s. In particular, it holds that 
 \[
 u\in C^0_w([0,T]; V_p) \qquad\P\text{-a.s.}
 \]
\end{thm}

Our next result gives sufficient conditions 
on the data so that the trajectories of $u$ are 
strictly separated from the potential barriers $\pm1$.
To this end, we will rely on the following assumption,
which can be seen as a generalisation of {\bf C} to suitable 
higher order derivates.

{\bf Assumption D}. There exists $\varsigma\in\enne\setminus\{0\}$ such that 
$\varsigma(p-d)> pd$ and 
\begin{align*}
  &(h_k)_{k\in\enne}\subset W_0^{\varsigma+2,\infty}(-1,1)\,,\\
  &C_{\H,\varsigma}^2:=\sum_{k=0}^\infty\norm{h_k}^2_{W^{\varsigma+2,\infty}(-1,1)}<+\infty\,.
\end{align*}

\begin{thm}[Separation property]
  \label{thm2}
  Assume {\bf A--B--C--D}, let 
  \[
  u_0\in V_p\cap L^\infty(\OO)\,, 
  \qquad F(u_0)\in L^1(\OO)\,,
  \]
  and suppose that $u_0$ is strictly separated from $\pm1$, namely that 
  \[
  \exists\,\delta_0\in(0,1):\quad
  \norm{u_0}_{L^\infty(\OO)}=1-\delta_0\,.
  \]
  Then $u$ is strictly separated from $\pm1$ almost surely, namely
  \[
  \P\left\{\exists\,\delta\in(0,\delta_0]:\;\sup_{(t,x)\in[0,T]\times\overline\OO}
  |u(t,x)|\leq1-\delta\right\}=1\,.
  \]
\end{thm}

Let us spend a few words on the regularity of $u$. More specifically, 
one has from Theorem~\ref{thm1+} that, for $\P$-almost every $\omega\in\Omega$,
$u(\omega)\in C^0([0,T]; H)\cap L^\infty(0,T;V_p)\subset C^0_w([0,T]; V_p)$.
This ensures that $u(\omega,t)$ makes sense as an element of $V_p$
for {\em every} $t\in[0,T]$. Moreover, in the setting of Theorem~\ref{thm2}
it readily follows that $p>d$, so that $V_p\embed C^0(\overline\OO)$:
consequently, it makes sense to evaluate $u(\omega)$ pointwise in $(t,x)$
for every $t\in[0,T]$ and $x\in\overline\OO$. It follows then that 
under the assumptions of Theorem~\ref{thm2}
there exists a measurable set $\Omega^*\in\cF$ with $\P(\Omega^*)=1$
such that
\[
  \norm{u(\omega)}_{L^\infty(Q)}=\sup_{(t,x)\in[0,T]\times\overline\OO}|u(\omega,t,x)|
  \qquad\forall\,\omega\in\Omega^*\,.
\]
Bearing this consideration in mind, 
Theorem~\ref{thm2} ensures that  
\[
\forall\,\omega\in\Omega^*\quad\exists\,\delta(\omega)\in(0,1):\quad
\norm{u(\omega)}_{L^\infty(Q)}\leq 1-\delta(\omega)\,.
\]
In other words, this means that almost every trajectory of $u$ is 
strictly separated from the boundary $\pm1$ of the potential $F$.
Let us stress, nonetheless, that the threshold of separation is not
uniform in $\omega$: this identifies in a natural way 
a random variable $\Lambda$ 
representing the ``amount'' of separation from $\pm1$.
Indeed, one can introduce 
\beq\label{def:delta}
  \Lambda:\Omega\to(0,\delta_0]\,, \qquad
  \Lambda(\omega):=
  \begin{cases}
  1-\norm{u(\omega)}_{L^\infty(Q)}\quad&\text{if } \omega\in\Omega^*\,,\\
  \delta_0\quad&\text{if } \omega\in\Omega\setminus\Omega^*\,,
  \end{cases}
\eeq
where $\delta_0\in(0,1)$ is defined as in Theorem~\ref{thm2}.
Clearly, $\Lambda$ is well-defined and valued in $(0,\delta_0]$ by Theorem~\ref{thm2}.
Also, $\Lambda$ is actually $\cF$-measurable: indeed, this follows from the fact that
$u:\Omega\to L^2(Q)$ is strongly measurable,
hence also $\cF/\cB(L^2(Q))$-measurable, and that the function 
\[
  \Phi_\infty:L^2(Q)\to[0,+\infty]\,, \qquad
  \Phi_\infty(v):=
  \begin{cases}
  \norm{v}_{L^\infty(Q)} \quad&\text{if } v\in L^\infty(Q)\,,\\
  +\infty \quad&\text{otherwise}\,,
  \end{cases}
\]
is $\cB(L^2(Q))/\cB(\erre)$-measurable.

We have seen that 
\eqref{def:delta} defines a random variable on $(\Omega,\cF)$
with values in $(0,\delta_0]$ almost everywhere, representing 
the magnitude of separation of the trajectories of $u$ from the barriers $\pm1$.
As we have pointed out before, $\Lambda$ is generally not constant in $\Omega$.
Of course, we have the trivial relations 
\[
  \P\{\Lambda\leq0\}=0 \qquad\text{and}\qquad\P\{\Lambda\leq\delta_0\}=1\,,
\]
meaning that the distribution of the random variable $\Lambda$
gives full measure to the interval $(0,\delta_0]$.
A natural question is then to precisely investigate the probability distribution of 
$\Lambda$, by studying the asymptotic behaviour of the measures of its
upper/lower level sets through the analysis 
of their rates of convergence.
Namely, for every fixed $\delta\in(0,\delta_0)$ close to $0$,
we aim at giving an estimation of the probability 
\[
\P\{\omega\in\Omega:
\Lambda(\omega)\leq\delta\}=\P\left\{\omega \in \Omega: \norm{u(\omega)}_{L^\infty(Q)}\geq1-\delta\right\}\,,
\]
and studying its behaviour as $\delta\searrow0$.
More specifically, for every $\delta\in(0,\delta_0)$,
the probability $\P\{\Lambda\leq\delta\}$ gives a 
quantitative measure of the trajectories that are 
separated from $\pm1$ of less than the threshold $\delta$.
Clearly, by Theorem~\ref{thm2} is necessarily holds that 
\[
  \lim_{\delta\searrow0}\P\{\Lambda\leq\delta\}=0\,.
\]
The aim of the following result is to give an estimate of the exact rate of convergence 
for such probability as $\delta\searrow0$.

\begin{thm}[Probability of separation]
  \label{thm3}
  Assume the setting of Theorem~\ref{thm2}. 
  Then, for every $\alpha\in (d/\varsigma, 1-d/p)$
  there exist two constants $L>0$ and $\delta_*\in(0,\delta_0)$, 
  depending only on $\alpha$, $\varsigma$, $p$, $T$, $\OO$, $F$, $\H$, and $u_0$,
  such that, setting $\rho:=p\frac{\varsigma - d/\alpha}{p + d/\alpha}$,
  \[
  \P\{\Lambda \leq \delta\} = 
  \P\left\{\sup_{(t,x)\in[0,T]\times\overline\OO}
  |u(t,x)|\geq1-\delta\right\}\leq  \exp\left(-L\delta^{-\rho}\right) \qquad
  \forall\,\delta\in(0,\delta_*)\,.
  \]
\end{thm}

Eventually, our last result is concerned with the investigation of the effect 
of the noise on the separation principle with respect to the deterministic equation.
More precisely, we consider here the following family of parameter-dependent problems
\begin{align}
  \label{1_eps}
  \d u_\eps - \Delta_p u_\eps\,\d t + F'(u_\eps)\,\d t 
  = \sqrt\eps\mathcal H(u_\eps)\,\d W \qquad&\text{in } (0,T)\times\OO\,,\\
  \label{2_eps}
  u_\eps=0 \qquad&\text{in } (0,T)\times\partial\OO\,,\\
  \label{3_eps}
  u_\eps(0)=u_0 \qquad&\text{in } \OO\,,
\end{align}
indexed with respect to $\eps\in[0,1]$.
The choice $\eps=0$ yields the corresponding deterministic equation,
which is well-known (details are given in Section~\ref{sec:proof4}) 
to admit a unique solution $\bar u$, 
constant in $\Omega$. Also, 
if the initial datum satisfies the classical condition $\|u_0\|_{L^\infty(\OO)}=1-\delta_0$ with $\max\{|r_F|, |R_F|\}\leq1-\delta_0$
(this is automatically true when $F'$ is non-decreasing for example),
the deterministic solution $\bar u$
is separated from the barriers of at least the deterministic threshold $\delta_0$, i.e.
\begin{align}
\label{bar1}
&\bar u\in L^\infty_\cP(\Omega; C^0([0,T]; H)\cap L^\infty(0,T; V_p))\,,\\
\label{bar2}
&\Delta_p\bar u, F'(\bar u) \in L^\infty_\cP(\Omega; L^2(0,T; H))\,,\\
\label{bar3}
&\sup_{(t,x)\in[0,T]\times\overline\OO}|\bar u(\omega,t,x)|= 
1-\delta_0\qquad\P\text{-a.e.~}\omega\in\Omega\,.
\end{align}
For every $\eps\in(0,1)$, Theorems~\ref{thm1+}--\ref{thm2} ensure
the existence and uniqueness of a process $u_\eps$
which is separated almost surely from the barriers $\pm1$
of at least a random threshold $\Lambda_\eps:\Omega\to(0,\delta_0]$, i.e.
\begin{align*}
&u_\eps\in L^\ell_\cP(\Omega; C^0([0,T]; H)\cap L^\infty(0,T; V_p))
\quad\forall\,\ell\in[1,+\infty)\,,\\
&\Delta_pu_\eps, F'(u_\eps) \in L^\ell_\cP(\Omega; L^2(0,T; H))
\quad\forall\,\ell\in[1,+\infty)\,,\\
&\sup_{(t,x)\in[0,T]\times\overline\OO}|u_\eps(\omega, t,x)|= 1-\Lambda_\eps(\omega)
\qquad\P\text{-a.e.~}\omega\in\Omega\,.
\end{align*}
It would be relevant to prove a convergence result of the random separation thresholds
$(\Lambda_\eps)_\eps$ towards the constant $\delta_0$ as $\eps\downarrow0$. 
Our last result answers this question in the sense of convergence in probability:
also, we are able to provide an exponential estimate of the convergence rate.
\begin{thm}
  \label{thm4}
  Assume the setting of Theorem~\ref{thm2} and 
  that $\max\{|r_F|, |R_F|\}\leq1-\delta_0$. Then, it holds for every $\eta\in(0,\delta_0)$ that 
  \[
  \lim_{\eps\downarrow0}\P\left\{|\Lambda_\eps - \delta_0|\geq\eta\right\}=0\,.
  \]
  More precisely, for every $\alpha\in(d/\varsigma, 1-d/p)$
  there exists a function $N:(0,\delta_0)\to(0,+\infty)$, only depending on
  $\alpha$, $\varsigma$, $p$, $T$, $\OO$, $F$, $\H$, and $u_0$, such that
  \begin{alignat*}{2}
  &\limsup_{\eps\downarrow0}\eps\ln\P\left\{|\Lambda_\eps-\delta_0|>\eta\right\}
  \leq - N(\delta_0-\eta) \qquad\forall\,\eta\in(0,\delta_0) \,.
  \end{alignat*}
\end{thm}


\section{Well-posedness}
\label{sec:proof1}
The proof of the well-posedness we present here adapts and extends the main ideas contained in \cite{bertacco} to the case of $p$-Laplacian operator, $p \in [2,+\infty)$, and to general singular potentials. 

For every $\lambda>0$, let $\beta_\lambda:\erre\to\erre$ 
be the Yosida approximation of $\beta$.
Let also $\hat\beta: \erre \to [0,+\infty]$ be 
the unique proper, convex, lower semicontinuous function such that 
$\partial\hat\beta=\beta$ and $\hat\beta(0)=0$, and 
$\hat\beta_\lambda:\erre\to[0,+\infty)$
the associated Moreau-Yosida regularization, 
i.e.~$\hat\beta_\lambda(r):= \int_0^r \beta_\lambda(s)\, \d s$, $r\in\erre$.
We set 
\[
  F_\lambda:\erre\to[0,+\infty)\,, \qquad
  F_\lambda(r):=F(0) + \hat\beta_\lambda(r) - \frac{C_F}2|r|^2\,, \quad r\in\erre\,,
\]
so that $F_\lambda$ is of class $C^2$ with 
\[ 
  F_\lambda'(r)=\beta_\lambda(r) - C_Fr\,, \quad r\in\erre\,.
\]
Let us also set $J_\lambda:=(I+\lambda\beta)^{-1}:\erre\to(-1,1)$ as the resolvent of $\beta$,
and define
\[
  \mathcal H_\lambda:H\to\cL^2(U,H)\,, \qquad
  \H_\lambda(v):=\H(J_\lambda(v))\,, \quad v\in\ H\,.
\]
Since $J_\lambda$ is non expansive, 
it is immediate to check that $\mathcal H_\lambda$ is $C_\H$-Lipschitz-continuous.

For every $\lambda>0$, we consider the approximated problem 
\begin{align*}
  \d u_\lambda - \Delta_p u_\lambda\,\d t + F_\lambda'(u)\,\d t 
  = \mathcal H_\lambda(u)\,\d W \qquad&\text{in } (0,T)\times\OO\,,\\
  u_\lambda=0 \qquad&\text{in } (0,T)\times\partial\OO\,,\\
  u_\lambda(0)=u_0 \qquad&\text{in } \OO\,.
\end{align*}
The classical variational theory by Pardoux \cite{Pard} 
and Krylov--Rozovskii \cite{KR-SPDE} ensures that there exists a 
unique approximated solution
\[
  u_\lambda \in L^2_\cP(\Omega; C^0([0,T]; H))\cap L^p_\cP(\Omega; L^p(0,T; V_p))\,,
\]
in the sense that the following equality is satisfied in $V_p^*$  
\[
  u_\lambda(t) - \int_0^t\div\gamma_p(\nabla u_\lambda(s))\,\d s + \int_0^tF'_\lambda(u(s))\,\d s = 
  u_0 + \int_0^t\H_\lambda(u_\lambda(s))\,\d W(s) \quad\forall\,t\in[0,T]\,, \quad\P\text{-a.s.}
\]

\subsection{First estimate}  
It\^o's formula for the square of the $H$-norm and assumptions {\bf A--B--C} yield
\begin{align*}
  \frac12\norm{u_\lambda(t)}_H^2 + \int_{Q_t}|\nabla u_\lambda|^p
  \leq\frac12\norm{u_0}_H^2 + C_F\int_{Q_t}|u_\lambda|^2
  +\frac{C_\H^2}2t + \int_0^t\left(u_\lambda(s),\H_\lambda(u_\lambda(s))\,\d W(s)\right)_H
\end{align*}
for every $t\in[0,T]$, $\P$-almost surely. By the Burkholder-Davis-Gundy (BDG) and Young inequalities,
for every $\ell\geq2$ we have
\begin{align*}
  \E\sup_{r\in[0,t]}
  \left|\int_0^r\left(u_\lambda(s),\H_\lambda(u_\lambda(s))\,\d W(s)\right)_H\right|^{\ell/2}
  &\leq M_\ell\E\left(\int_0^t\norm{u_\lambda(s)}_H^2
  \norm{\H_\lambda(u_\lambda(s))}_{\cL^2(U,H)}^2\,\d s\right)^{\ell/4}\\
  &\leq \sigma\E\sup_{r\in[0,t]}\norm{u_\lambda(r)}^\ell_{H}+
  M_{\ell,\sigma} C_\H^{\ell} |\OO|
\end{align*}
for every $\sigma>0$ and for some constants 
$M_\ell, M_{\ell,\sigma}>0$ independent of $\lambda$.
Taking power $\ell/2$, choosing 
$\sigma>0$ sufficiently small, rearranging the terms,
and using the Gronwall lemma, we infer that 
possibly renominating $M_\ell>0$ independent of $\lambda$ it holds that
\beq
  \label{est1}
  \norm{u_\lambda}_{L^\ell_\cP(\Omega; C^0([0,T]; H))\cap 
  L^{p\ell/2}_\cP(\Omega; L^p(0,T; V_p))} \leq M_\ell\,.
\eeq

\subsection{Second estimate} 
Let $\mathcal F_\lambda: V_p \to \erre$ be defined as $\mathcal F_\lambda(u):= \int_\OO F_\lambda(u)$. Thanks to Assumption {\bf B},  $\mathcal F_\lambda$ is Fr\'echet differentiable with derivatives given by 
\begin{align*}
 D\mathcal F_\lambda(u)[y] &= \int_\OO F_\lambda'(u) y \\
 D^2\mathcal F_\lambda(u)[y_1, y_2] &= \int_\OO F_\lambda''(u) y_1 y_2\,,
 \end{align*}
 for every $u,y,y_1,y_2 \in V_p$.
Recalling that $F_\lambda$ is of class $C^2$ with Lipschitz derivative,
It\^o formula for $\mathcal F_\lambda$ yields 
\begin{equation}\label{ito-F}
\begin{aligned}
  &\int_\OO F_\lambda(u_\lambda(t)) + \int_{Q_t}F_\lambda''(u_\lambda)|\nabla u_\lambda|^p
  +\int_{Q_t}|F_\lambda'(u_\lambda)|^2\\
  &=\int_\OO F_\lambda(u_0) + \frac12\int_0^t\sum_{k=0}^\infty
  \int_\OO F_\lambda''(u_\lambda(s))|h_k(J_\lambda(u_\lambda(s)))|^2\,\d s\\
  &\qquad+ \int_0^t\left(F_\lambda'(u_\lambda(s)),\H_\lambda(u_\lambda(s))\,\d W(s)\right)_H
\end{aligned}
\end{equation}
for every $t\in[0,T]$, $\P$-almost surely.
Since $\beta_\lambda'$ is increasing it holds that $F_\lambda''(u_\lambda) = \beta_\lambda'(u_\lambda) - C_F \geq - C_F$, so that 
\[ \int_{Q_t}F_\lambda''(u_\lambda)|\nabla u_\lambda|^p \geq -C_F \int_{Q_t}|\nabla u_\lambda|^p \]
Furthermore, recalling that $\beta_\lambda(u_\lambda) = \beta(J_\lambda(u_\lambda))$ and that $J_\lambda:\erre \to (-1,1)$ is non-expansive, we have that 
$F_\lambda''(u_\lambda) = \beta'(J_\lambda(u_\lambda))J_\lambda'(u_\lambda) 
- C_F \leq
\beta'(J_\lambda(u_\lambda)) - C_F=
 F''(J_\lambda(u_\lambda))$.
This is useful to estimate the term
\begin{align*}
\int_0^t\sum_{k=0}^\infty 
\int_\OO F_\lambda''(u_\lambda(s))|h_k(J_\lambda(u_\lambda(s)))|^2\, \d s 
&\leq  \int_0^t\sum_{k=0}^\infty \int_\OO 
F''(J_\lambda (u_\lambda(s)))\left|h_k(J_\lambda(u_\lambda(s)))\right|^2\, \d s \\
&\leq |\OO| \int_0^t\sum_{k=0}^\infty \norm{F''h_k^2}_{L^\infty(-1,1)} \, \d s \\
&\leq C_\cH^2 |\OO| t \, .
\end{align*}
For what concerns the stochastic integral, Assumption {\bf C}, the BDG and Young inequalities yield
\begin{align*}
& \E \sup_{t \in [0,T]} \left|\int_0^t \left( F_\lambda'(u_\lambda(s)), 
\cH_\lambda(u_\lambda(s))\, \d W(s) \right)_H \right|^{\ell/2} \\
& \leq M_2\E \left( \int_0^T \| F_\lambda'(u_\lambda(s)) \|_H^2 \| \cH_\lambda(u_\lambda(s)) \|^2_{\cL^2(U,H)} \, \d s \right)^{\ell/4} \\
& \leq  \sigma\E\left( \int_0^T \| F_\lambda'(u_\lambda(s)) \|_H^2 \, \d s\right)^{\ell/2} +
 C_\H^\ell|\OO| M_{2,\sigma} \, , 
\end{align*}  
for every $\sigma >0$ and for some constants $M_2, M_{2,\sigma}>0$ independent of $\lambda$.
Summing up and possibly renominating $M_\ell$ again, we get that 
\[ 
\| F_\lambda'(u_\lambda) \|_{L^\ell_\cP(\Omega;L^2(0,T;H))}^\ell
\leq M_\ell(1+\E\|\nabla u_\lambda\|_{L^p(0,T; L^p(\OO))}^{p\ell/2})\]
and,  thanks to estimate \eqref{est1}, 
\begin{equation}\label{est2}
\| F_\lambda'(u_\lambda) \|_{L^\ell_\cP(\Omega;L^2(0,T;H))}  
 + \| \beta_\lambda(u_\lambda) \|_{L^\ell_\cP(\Omega;L^2(0,T;H))}  \leq M_\ell \, .
\end{equation}

\subsection{Cauchy property} 
For every $\lambda > \varepsilon > 0$ let us write the equation for the difference $u_\lambda - u_\varepsilon$: for every $v \in V_p$, for every $t \in [0,T]$, $\P$-a.s. it holds 
\begin{equation*}
\begin{aligned}
  \int_\OO &\left( u_\lambda(t) - u_\varepsilon(t) \right) v + \int_0^t \int_\OO \left( \gamma_p(\nabla u_\lambda(s)) - \gamma_p(\nabla u_\varepsilon(s)) \right) \cdot\nabla v\,\d s \\
  &+ \int_0^t \int_\OO  \left( F_\lambda'(u_\lambda(s)) - F_\varepsilon'(u_\varepsilon(s)) \right) v \,\d s 
  = \int_\OO \left(\int_0^t  \left( \H_\lambda(u_\lambda(s)) - \H_\varepsilon(u_\varepsilon) \right)\,\d W(s) \right) v \, .
\end{aligned}
\end{equation*}
It\^o formula for the square of the $H$-norm yields
\begin{equation*}
\begin{aligned}
  &\frac{1}{2}\|u_\lambda(t) - u_\varepsilon(t)\|_H^2 + 
  \int_0^t \big( \gamma_p(\nabla u_\lambda(s)) - 
  \gamma_p(\nabla u_\varepsilon(s)), 
   \nabla u_\lambda(s) - \nabla u_\varepsilon(s) \big)_{L^2(\OO)^d} \,\d s \\
  &\qquad+ \int_0^t \left( F_\lambda'(u_\lambda(s)) - F_\varepsilon'(u_\varepsilon(s)), 
   u_\lambda(s) - u_\varepsilon(s) \right)_H \,\d s \\
  &=  \frac{1}{2} \int_0^t \sum_{k=0}^\infty \int_\OO 
  | h_k(J_\lambda(u_\lambda(s))) - h_k(J_\varepsilon(u_\varepsilon(s))) |^2 \, \d s \\
  &\qquad+\int_0^t  \big( u_\lambda(s) - u_\varepsilon(s), (\H_\lambda(u_\lambda(s)) - \H_\varepsilon(u_\varepsilon(s)) )\,\d W(s) \big)_H \, .
\end{aligned}
\end{equation*}
Strong monotonicity of the p-Laplacian (see \cite[Lem.~2.1]{simon_p}) guarantees that
 \[ \big( \gamma_p(\nabla u_\lambda(s)) - \gamma_p(\nabla u_\varepsilon(s)),  \nabla u_\lambda(s) - \nabla u_\varepsilon(s) \big)_{L^2(\OO)^d} \geq c_p \| u_\lambda(s)-u_\varepsilon(s)\|_{V_p}^p. \]
Exploiting the monotonicity of $\beta_\lambda$, writing 
$u_\lambda = \lambda \beta_\lambda(u_\lambda) + J_\lambda(u_\lambda)$ 
(and the analogous for $u_\varepsilon$) and using Young inequality we get  
\begin{align*}
\left( \beta_\lambda(u_\lambda) - \beta_\varepsilon(u_\varepsilon),  u_\lambda - u_\varepsilon \right)_H &\geq \lambda  \| \beta_\lambda(u_\lambda) \|_H^2 + \varepsilon \| \beta_\varepsilon(u_\varepsilon)\|_H^2 
- (\varepsilon + \lambda)(\beta_\lambda(u_\lambda), \beta_\varepsilon(u_\varepsilon))_H \\
&\geq -\frac{(\lambda + \varepsilon)}{2} \left( \| \beta_\lambda(u_\lambda)\|_H^2 + \| \beta_\varepsilon(u_\varepsilon)\|^2_H \right),
\end{align*}
Since $F_\lambda'(u_\lambda) = \beta_\lambda(u_\lambda) - C_F u_\lambda $, the above estimate and \eqref{est2} entail
\begin{align*}
&\E\left(\int_0^t \left( F_\lambda'(u_\lambda(s)) - F_\varepsilon'(u_\varepsilon(s)),  u_\lambda(s) - u_\varepsilon(s) \right)_H \,\d s\right)^{\ell/2} \\
&\geq - K(\lambda + \varepsilon) 
- C_F\E \left(\int_0^t
 \| (u_\lambda - u_\varepsilon)(s) \|_H^2 \, \d s\right)^{\ell/2}\,.
 \end{align*}
To control the It\^o correction term and the stochastic integral it is useful to estimate 
\begin{align*}
&\| J_\lambda(u_\lambda) - J_\varepsilon(u_\varepsilon) \|_H^2 \leq 
\| J_\lambda(u_\lambda) - J_\lambda(u_\varepsilon) \|_H^2 + \| J_\lambda(u_\varepsilon) - J_\varepsilon(u_\varepsilon) \|_H^2 \\
&\leq \| u_\lambda - u_\varepsilon \|_H^2 + \| (J_\lambda(u_\varepsilon) - u_\varepsilon) - (J_\varepsilon(u_\varepsilon) - u_\varepsilon) \|_H^2 \\
&\leq \| u_\lambda - u_\varepsilon \|_H^2 + 2(\lambda^2 \|\beta_\lambda(u_\varepsilon)\|_H^2 + \varepsilon^2 \|\beta_\varepsilon(u_\varepsilon)\|_H^2)\\
&\leq \| u_\lambda - u_\varepsilon \|_H^2 + 2(\lambda^2 + \varepsilon^2)\| \beta_\varepsilon(u_\varepsilon) \|^2_H\,,
\end{align*}
where we used the non-expansivity of $J_\lambda$ and the fact that $|\beta_\lambda|$ is increasing as $\lambda \downarrow 0$.
Whence,  thanks to estimate \eqref{est2},  
\begin{align*}
&\frac{1}{2}\E \left(\int_0^t \sum_{k=0}^\infty \int_\OO | h_k(J_\lambda(u_\lambda(s))) - h_k(J_\varepsilon(u_\varepsilon(s))) |^2 \, \d s\right)^{\ell/2} \\
&\leq \frac{C_\cH}2\E\left( \int_0^t \| J_\lambda(u_\lambda(s)) - J_\varepsilon(u_\varepsilon(s)) \|_H^2 \, \d s\right)^{\ell/2} \\ 
&\lesssim \E\left( \int_0^t \| (u_\lambda - u_\varepsilon)(s) \|_H^2 \, \d s\right)^{\ell/2} + K(\lambda^2 + \varepsilon^2)
\end{align*}
and for every $\sigma >0$ there exists $M_\sigma$ independent of $\lambda$ such that 
\begin{align*}
\E &\sup_{r \in [0,t]} \left| \int_0^r  \big( u_\lambda(s) - u_\varepsilon(s), 
(\H_\lambda(u_\lambda(s)) - \H_\varepsilon(u_\varepsilon(s)) )\,\d W(s) \big)_H \right|^{\ell/2} \\
&\lesssim \E \left(\int_0^t \| u_\lambda(s) - u_\varepsilon(s) \|_H^2 \| \H_\lambda(u_\lambda(s)) - \H_\varepsilon(u_\varepsilon(s)) \|^2_{\cL^2(U,H)} \, \d s \right)^{\ell/4}\\
&\leq \sigma\E \left(\sup_{r \in [0,r]} \| u_\lambda(r) - u_\varepsilon(r) \|_H^2 \right)^{\ell/2}  + M_\sigma \E \left( \int_0^t \| J_\lambda(u_\lambda(s)) - J_\varepsilon(u_\varepsilon(s)) \|^2_{H} \, \d s \right)^{\ell/2}\\
&\leq \sigma\E \left(\sup_{r \in [0,t]} \| u_\lambda(r) - u_\varepsilon(r) \|_H^2 \right)^{\ell/2}  + M_\sigma \E \left(\int_0^t \| (u_\lambda - u_\varepsilon)(s) \|_H^2 \, \d s\right)^{\ell/2} +
 M_{\ell,\sigma}(\lambda^2 + \varepsilon^2)\,.
\end{align*}
Combining all the estimates, taking the supremum in time, power $\ell/2$, exploiting the arbitrariness of $\sigma >0$ and Gronwall's inequality we obtain
\begin{equation}\label{est3}
\norm{u_\lambda - u_\varepsilon}_{L^\ell_\cP(\Omega; C^0([0,T]; H))\cap 
  L^{p\ell/2}_\cP(\Omega; L^p(0,T; V_p))} \lesssim e^{CT} \left(\lambda + \varepsilon + \lambda^2 + \varepsilon^2 \right) \longrightarrow 0\, , \quad \text{ as } \lambda, \varepsilon \downarrow 0 \,.
\end{equation}  

\subsection{Existence of a solution}
We pass to the limit as $\lambda \downarrow 0$ to deduce the existence of  a variational solution in the form \eqref{variational_sol}.
From \eqref{est3} we deduce that $(u_\lambda)_{\lambda}$ is a Cauchy sequence so that  
\begin{equation}\label{conv1}
u_\lambda \to u \qquad \text{ in } L^\ell_\cP(\Omega; C^0([0,T]; H))\cap 
  L^{p\ell/2}_\cP(\Omega; L^p(0,T; V_p)) \,.
\end{equation}
Since the operator $-\Delta_p: V_p \to V_p^*$ is demicontinuous, from \eqref{conv1} it follows that 
\begin{equation}\label{conv2}
-\Delta_pu_\lambda \wto -\Delta_pu \qquad \text{ in } L^q_\cP(\Omega; L^q(0,T; V_p^*))\,. 
\end{equation} 
The uniform bound \eqref{est2} further guarantees that, up to extracting a subsequence $(\lambda_k)_k$,  
\begin{equation*}
F_{\lambda_k}'(u_{\lambda_k}) \wto \xi \qquad \text{ in } L^\ell_\cP(\Omega; L^2(0,T; H))\,, 
\end{equation*}
whence $\beta_{\lambda_k}(u_{\lambda_k}) \wto \xi + C_F u$ in $L^2_\cP(\Omega; L^2(0,T; H))$.
It remains to show that $\xi = F'(u)$. 
Noting that $J_\lambda(u_\lambda)\to u$ in $L^2_\cP(\Omega; L^2(0,T; H))$
by definition of Yosida approximation and the bound \eqref{est2},
the weak-strong closure of the maximal monotone graph $\beta$ (see \cite[Ch.~2]{barbu-monot})
yields that $\xi+ C_Fu=\beta(u)$, hence in particular that 
\begin{equation}\label{conv3}
F_{\lambda_k}'(u_{\lambda_k}) \wto F'(u) \qquad \text{ in } L^2_\cP(\Omega; L^2(0,T; H))\,.
\end{equation}
Notice that the regularity $F'(u) \in L^2(\Omega; L^2(0,T; H))$ readily implies the bound 
\begin{equation*}
\| u \|_{L^\infty(\OO)} \leq 1\qquad \text{ for a.e.~}t \in [0,T]\,, \quad\P\text{-a.e.}
\end{equation*}
Moreover, the BDG inequality and the Lipschitz character of $\H$ yield the convergence of the stochastic integrals
\begin{equation*}
\E \sup_{t \in [0,T]}\left| \int_0^t \big( (\H_\lambda(u_\lambda(s)) - \H(u(s)) )  dW(s)\big)_H \right|^2 \lesssim \| J_\lambda(u_\lambda) - u \|^2_{L^2_\cP(\Omega; L^2(0,T); H))} \to 0\,,
\end{equation*} 
so that 
\begin{equation}\label{conv4}
\int_0^\cdot \H_\lambda(u_\lambda(s))dW(s) \to \int_0^\cdot \H(u(s))dW(s)\,, \qquad \text{ in } L^2(\Omega; C^0([0,T];H)\,. 
\end{equation} 
Combining the convergences in \eqref{conv1}-\eqref{conv4} we can pass to the limit in the approximated equation along the subsequence $(\lambda_k)_k$ with $k \to +\infty$.
Therefore, we conclude that 
$u$ is a variational solution of the Allen-Cahn equation in the sense of Theorem~\ref{thm1}.

\subsection{Uniqueness of solutions}
Let $u^1_{0}, u^2_0 \in H$ be two different initial data and denote with $u_1, u_2 \in L^p_\cP(\Omega; C^0([0,T]; H))\cap L^p_\cP(\Omega; L^p(0,T; V_p))$ the associated solutions of \eqref{variational_sol}.
Following the same strategy as in the proof of the third estimate it is easy to show that 
\begin{equation}\label{cont_dip}
\|u_1- u_2\|_{L^p_\cP(\Omega; C^0([0,T]; H))\cap L^p_\cP(\Omega; L^p(0,T; V_p))} \lesssim \| u_0^1 - u_0^2\|_{L^2(\Omega;H)}
\end{equation} 
Choosing $u_0^1 = u_0^2$, this readily implies pathwise  uniqueness of solutions to equation \eqref{variational_sol} under Assumptions {\bf A-B-C}.

\subsection{Refined existence}
Let us suppose now that $u_0 \in V_p$.
Define the map $\Phi: V_p \to \R$ as $\Phi_p(u):= \frac{1}{p}\|\nabla u\|_{L^p(\OO)}^p$. The first and second Fr\'echet derivative of $\Phi_p$ are given by
\begin{align*}
D\Phi_p(u)[y] &= \int_\OO \gamma_p(\nabla u) \cdot \nabla y \\
D^2\Phi_p(u)[y_1,y_2] &= (p-1)\int_\OO |\nabla u|^{p-2}\nabla y_1 \cdot \nabla y_2 \,,
\end{align*}
for every $u,y,y_1,y_2 \in V_p$.
Notice that $\|\nabla u\|_{L^p(\OO)}^p = \|\gamma_p(\nabla u)\|_{L^q(\OO)}^q$, where $q = \frac{p}{p-1}$.
The additional assumption $u_0\in V_p$
guarantees (see for example \cite{gess}) that 
$u_\lambda\in L^\infty(0,T; L^1(\Omega; V_p))$ and
$-\Delta_p u_\lambda\in L^2_\cP(\Omega; L^2(0,T; H))$.
Consequently, It\^o formula for $\Phi_p(u_\lambda)$
(see \cite[Prop.~3.3]{SS-doubly2}) yields after some easy manipulations
\begin{align*}
&\frac{1}{p}\int_\OO |\nabla u_\lambda(t)|^p + 
\int_{Q_t}|\Delta_pu_\lambda|^2  + 
\int_{Q_t}F_\lambda''(u_\lambda) |\nabla u_\lambda|^p = \frac{1}{p}\int_\OO |\nabla u_0|^p \\
&+ (p-1)\sum_{k=0}^\infty \int_{Q_t} |\nabla u_\lambda|^p |J_\lambda'(u_\lambda)|^2 |h_k'(J_\lambda(u_\lambda))|^2 + \int_0^t \big( \gamma_p(\nabla u_\lambda(s)) , \nabla\H_\lambda(u_\lambda(s)) \, \d W(s)\big)_{L^2(\OO)^d}\,.
\end{align*}
The third term on the l.h.s can be estimated as usual ($\beta_\lambda'(u_\lambda) \geq 0$) 
\begin{equation*}
\int_{Q_t} |\nabla u_\lambda|^p F_\lambda''(u_\lambda)
 \geq -C_F \int_{Q_t} |\nabla u_\lambda|^p\,.
\end{equation*}
Thanks to Assumption {\bf C} and the non-expansive character of $J_\lambda$ 
the trace term can be estimated as 
\begin{equation*}
\sum_{k=0}^\infty \int_{Q_t} |\nabla u_\lambda|^p |J_\lambda'(u_\lambda)|^2 |h_k'(J_\lambda(u_\lambda))|^2 \leq C_{\H} \int_{Q_t} |\nabla u_\lambda|^p\,.
\end{equation*}
Concerning the stochastic integral we have
\begin{align*}
\E &\sup_{t \in [0,T]} \left| \int_0^t \big( \gamma_p(\nabla u_\lambda(s)) , 
\nabla\H_\lambda(u_\lambda(s)) \, \d W(s)\big)_{L^2(\OO)^d} \right| \\
&= \E \sup_{t \in [0,T]} \left| \int_0^t \big( |\nabla u_\lambda(s)|^{\frac{p-2}{2}}
\nabla u_\lambda(s) , |\nabla u_\lambda(s)|^{\frac{p-2}{2}}
\nabla\H_\lambda(u_\lambda(s)) \, \d W(s)\big)_{L^2(\OO)^d} \right| \\
&\leq \E \left( \sup_{t \in [0,T]}\| |\nabla u_\lambda(t)|^{\frac{p-2}{2}}\nabla u_\lambda(t) \|_H^2
 \int_0^T \| |\nabla u_\lambda(s)|^{\frac{p-2}{2}}
 \nabla\H_\lambda(u_\lambda(s)) \|_{\cL^2(U,H)}^2\right)^{1/2}\\
&\leq \sigma \E \sup_{t \in [0,T]} \int_\OO|\nabla u_\lambda(t)|^p + 
M_\sigma \E \sum_{k=1}^\infty \int_{Q}
|\nabla u_\lambda|^p |h_k'(J_\lambda(u_\lambda))|^2 \\
&\leq \sigma \E \sup_{t \in [0,T]} \int_\OO|\nabla u_\lambda(t)|^p 
+ M_\sigma C_{\H}\E \int_{Q} |\nabla u_\lambda|^p
\end{align*}
Taking the supremum in time and expectation, the arbitrariness of $\sigma >0$ and 
the estimate \eqref{est1} yield
\begin{equation}\label{est4}
\|u_\lambda\|_{L^p_w(\Omega; L^\infty(0,T; V_p))}
+\|\Delta_p u_\lambda\|_{L^2_\cP(\Omega; L^2(0,T; H))}\leq K\,.
\end{equation}
Once the estimate \eqref{est4} is established, the proof of Theorem \ref{thm1+} relies on standard arguments.
Indeed, one readily gets that 
\[
  u_\lambda \wstarto u \qquad\text{in } L^p_w(\Omega; L^\infty(0,T; V_p))
\]
 and
\[
  -\Delta_p u_\lambda \wto -\Delta_p u
  \qquad\text{in } L^2_\cP(\Omega; L^2(0,T; H))\,.
\]
Hence $u$ belongs to  $L^p_w(\Omega; L^\infty(0,T; V_p))$ with 
$\Delta_p u \in L^2_\cP(\Omega; L^2(0,T; H_{\div}))$ and satisfies \eqref{var_sol+}.
This proves Theorem~\ref{thm1+}.

\section{Separation property}
\label{sec:proof2}
We present here the proof of Theorem~\ref{thm2}, which 
 is divided in two main steps. In the first step we prove that,
almost surely in $\Omega$, it holds $|u(t, x)| < 1$ for every $(t, x) \in (0, T) \times \OO$,
namely that $u$ cannot touch the barriers $\pm1$ at any point in space and time.
Afterwards, in the second step, we show that, for almost every $\omega \in \Omega$
it exists a threshold $\delta = \delta(\omega) > 0$
 such that $\|u(\omega)\|_{L^{\infty}(Q)} \leq 1 - \delta(\omega)$. 
In order to show this, we will rely on suitable pathwise estimates
and H\"older regularity of the solution. Throughout the section, 
the assumptions of Theorem~\ref{thm2} are in order.

\subsection{Pathwise estimates}
Given $\varsigma\in\mathbb{N} \setminus \{0\}$ 
as in the assumption {\bf D}, 
in order to ``measure'' how much $u$ is concentrated around $\pm1$
we use the following convex function 
\begin{equation}\label{Gs}
G_\varsigma : (-1,1) \to \mathbb{R}\,, \qquad
G_\varsigma (r) := \frac{1}{(1-r^2)^\varsigma}\,, \quad r \in \mathbb{R}\,,
\end{equation}
and also the corresponding functional
\begin{equation*}
\mathcal{G}_\varsigma: H \to [0,+\infty]\,,\qquad
\mathcal{G}_\varsigma(v) :=
\begin{cases}
 \displaystyle
 \int_{\OO} G_\varsigma(v) \quad&\text{if } G_\varsigma(v) \in L^1(\OO)\,,\\
 +\infty \quad&\text{otherwise}\,.
 \end{cases}
\end{equation*}
A direct computation shows that $G_\varsigma', G_\varsigma'':(-1,1)\to\erre$ are given by
\[
  G_\varsigma'(r)=2\varsigma\frac{r}{(1-r^2)^{\varsigma+1}}\,,
  \qquad
  G''_\varsigma(r)=2\varsigma\frac{(2\varsigma+1)r^2+1}{(1-r^2)^{\varsigma+2}}\,,
  \qquad r\in(-1,1)\,.
\]

\begin{lem}
  \label{lem:path1}
  There exists $K_1>0$, only depending of $\varsigma$, $p$, $T$, $\OO$, $F$,
  $\H$, and $u_0$,
  such that 
  \beq\label{est_Gs}
  \sup_{t\in[0,T]}\int_\OO G_\varsigma(u(t))
  \leq K_1 + \sup_{t\in[0,T]}\int_0^t
  \left(G'_\varsigma(u(s)), \H(u(s))\,\d W(s)\right)_H<+\infty
  \qquad\P\text{-a.s.}
  \eeq
  where the stochastic integral is well-defined as
  $(G'_\varsigma(u),\H(u)\cdot)_H\in L^2_\cP(\Omega; L^2(0,T; \cL^2(U,\erre)))$.
\end{lem}
\begin{proof}[Proof of Lemma~\ref{lem:path1}]
Ideally, we would like to write It\^o's formula for $\mathcal{G}_\varsigma(u)$:
however, since $G'_\varsigma$ is not Lipschitz continuous and $G''_\varsigma$ is not bounded, we cannot relay on standard It\^o's formulas in the variational setting. Hence, we use an approximation of the function $G_\varsigma$ in order to recover such regularity. For example, since $G'_\varsigma$ is monotone increasing and continuous, then we can identify it with a maximal monotone graph. Therefore, for every $\lambda \in (0, 1)$, we can consider its Yosida approximation $G'_{\varsigma, \lambda}: \mathbb{R} \to \mathbb{R}$ and define 
$G_{\varsigma, \lambda}: \mathbb{R}\to\mathbb{R}$ as
\begin{equation*}
G_{\varsigma, \lambda}(r):=1+ \int_0^r G'_{\varsigma, \lambda}(y) \,\d y\,, \quad r \in \mathbb{R}\,.
\end{equation*}
In this way we have that $G_{\varsigma, \lambda} \in C^2(\mathbb{R})$
is exactly the Moreau-Yosida regularisation of $G_\varsigma$, and 
 is such that $G'_{\varsigma, \lambda}$ is Lipschitz-continuous and 
 $G''_{\varsigma, \lambda}$ is continuous and bounded. Therefore, letting
\begin{equation*}
\mathcal{G}_{\varsigma, \lambda}: H \to \mathbb{R}\,, \qquad
\mathcal{G}_{\varsigma, \lambda}(v) := \int_D G_{\varsigma, \lambda}(v)\,, \quad v \in H\,,
\end{equation*}
we can write It\^o's formula for $\mathcal{G}_{\varsigma, \lambda}(u)$ and then,
by passing to the limit for $\lambda \downarrow 0$, we can recover an It\^o's inequality for 
$\mathcal{G}_{\varsigma}(u)$ by lower semicontinuity. Indeed, for every 
$t \in [0, T]$, $\P$-almost surely it holds that 
\begin{equation*}
\begin{aligned}
  &\int_\OO G_{\varsigma,\lambda}(u(t)) +
   \int_{Q_t}G''_{\varsigma,\lambda}(u)|\nabla u|^p
  +\int_{Q_t}G_{\varsigma,\lambda}'(u) F'(u)\\
  &=\int_\OO G_{\varsigma,\lambda}(u_0) + \frac12\int_0^t\sum_{k=0}^\infty
  \int_\OO G''_{\varsigma,\lambda}(u(s))|h_k(u(s))|^2\,\d s
  + \int_0^t\left(G_{\varsigma,\lambda}'(u(s)),\H(u(s))\,\d W(s)\right)_H\,.
\end{aligned}
\end{equation*}
Now, the second term on the left-hand side is nonnegative 
by convexity of $G_{\varsigma, \lambda}$.
For the third term on the left-hand side, by assumption {\bf B},
by monotonicity,
and the fact that $G'_{\varsigma,\lambda}(0)=0$ we have that 
\begin{align*}
&\int_{Q_t} G'_{\varsigma,\lambda}(u) F'(u)  \\
&=\int_{Q_t\cap\{u\in(-1,r_F)\}}
G'_{\varsigma,\lambda}(u) F'(u) 
+\int_{Q_t\cap\{u\in[r_F,R_F]\}}
G'_{\varsigma,\lambda}(u) F'(u) +\int_{Q_t\cap\{u\in(R_F,1)\}}
G'_{\varsigma,\lambda}(u) F'(u)  \\
& \geq \int_{Q_t\cap\{u\in[-r_F,R_F]\}}
G'_{\varsigma,\lambda}(u) F'(u)\geq
 - |Q|\max_{r\in[r_F,R_F]}|G_\varsigma'(r)|\max_{r\in[r_F,R_F]}|F'(r)|\,.
\end{align*}
Furthermore, on the right-hand side we have by assumption on the initial datum that 
\[
\int_\OO G_{\varsigma,\lambda}(u_0)\leq \int_\OO G_{\varsigma}(u_0)
\leq \frac{|\OO|}{\delta_0^{2\varsigma}}\,.
\]
Putting together this information we deduce that 
\begin{equation*}
\begin{aligned}
  \int_\OO G_{\varsigma,\lambda}(u(t)) 
  &\leq \frac{|\OO|}{\delta_0^{2\varsigma}} +
   |Q|\max_{r\in[r_F,R_F]}|G_\varsigma'(r)|\max_{r\in[r_F,R_F]}|F'(r)|\\
 &+ \frac12\int_0^t\sum_{k=0}^\infty
  \int_\OO G''_{\varsigma,\lambda}(u)|h_k(u)|^2
  + \int_0^t\left(G_{\varsigma,\lambda}'(u(r)),\H(u(r))\,\d W(r)\right)_H\,.
\end{aligned}
\end{equation*}
Let us handle the trace term. To this end, we exploit assumption {\bf D}:
indeed, for every $k \in \mathbb{N}$, thanks to Taylor's Theorem with integral reminder,
 it holds for every $r\in(-1,1)$ that
\begin{align*}
h_k(r) &= \sum_{j = 0}^{\varsigma+1} 
\frac{h_k^{(j)}(1)}{j!}
 (r-1)^{j} - \frac1{(\varsigma+1)!}\int_r^1 
 h^{(\varsigma+2)}_k(y)(r-y)^{\varsigma+1} \, \d y\,,\\
 h_k(r) &= \sum_{j = 0}^{\varsigma+1} 
\frac{h_k^{(j)}(-1)}{j!}
 (r+1)^{j} + \frac1{(\varsigma+1)!}\int_{-1}^r 
 h^{(\varsigma+2)}_k(y)(r-y)^{s+1} \, \d y\,.
\end{align*}
Recalling assumption {\bf D}, we have for every $r\in(-1,1)$ and $k\in\enne$ that 
\begin{align*}
  |h_k(r)| &\leq \frac{\|h_k\|_{W^{\varsigma+2, \infty}(-1, 1)}}{(\varsigma+2)!} |r-1|^{\varsigma+2}\,,\\
  |h_k(r)| &\leq \frac{\|h_k\|_{W^{\varsigma+2, \infty}(-1, 1)}}{(\varsigma+2)!} |r+1|^{\varsigma+2}\,.
\end{align*}
Now, the non-expansivity of the resolvent of $G_\varsigma'$
and the fact that $r\mapsto G_\varsigma''(|r|)$, $r\in(-1,1)$, is non-increasing
imply that $G''_{\varsigma,\lambda}(u)\leq G''_\varsigma(u)$.
Therefore, exploiting the form of $G''_\varsigma$
and the fact that $|u|\leq1$ almost everywhere, we infer that
\begin{align*}
&\frac{1}{2} \int_0^t  \sum_{k=0}^{\infty} \int_{\OO} 
G''_{\varsigma,\lambda}(u(s)) |h_k(u(s))|^2\,\d s \\
& \leq 2\varsigma(\varsigma+1)
\int_0^t \sum_{k=0}^{\infty} \int_{\OO} 
\frac{|h_k(u(s))|^2}{(1-u(s)^2)^{\varsigma+2}}\,\d s \\
& \leq \frac{2\varsigma(\varsigma+1)}{(\varsigma+2)!^2} 
\sum_{k=0}^{\infty} \|h_k\|^2_{W^{\varsigma+2, \infty}(-1, 1)}
\int_{Q_t} \left( \frac{|u-1| |u+1|}{|u-1| |u+1|} \right)^{\varsigma+2} \\
& \leq \frac{2\varsigma(\varsigma+1)}{(\varsigma+2)!^2} 
C^2_{\H, \varsigma} |Q|\,.
\end{align*}
Consequently, by choosing
\beq\label{K1}
  K_1:=\frac{|\OO|}{\delta_0^{2\varsigma}} +
   |Q|\max_{r\in[r_F,R_F]}|G_\varsigma'(r)|\max_{r\in[r_F,R_F]}|F'(r)|
   +\frac{2\varsigma(\varsigma+1)}{(\varsigma+2)!^2} 
  C^2_{\H, \varsigma} |Q|\,,
\eeq
we get exactly 
\beq\label{aux_path1}
  \int_\OO G_{\varsigma,\lambda}(u(t))
  \leq K_1 + \int_0^t
  \left(G'_{\varsigma,\lambda}(u(s)), \H(u(s))\,\d W(s)\right)_H
  \qquad\forall\,t\in[0,T]\,,\quad\P\text{-a.s.}
\eeq
Now, by the BDG inequality we have 
\begin{align*}
  &\E\sup_{r\in[0,t]}\int_0^r\left(G'_{\varsigma,\lambda}(u(s)), \H(u(s))\,\d W(s)\right)_H
  \leq C\E\left(\sum_{k=0}^\infty\int_{Q_t}
  G'_{\varsigma,\lambda}(u)^2|h_k(u)|^2\right)^{1/2}\,,
\end{align*}
where
\begin{align*}
  G'_{\varsigma,\lambda}(u)^2|h_k(u)|^2
  &\leq 4\varsigma^2\frac{|h_k(u)|^2}{(1-u^2)^{2\varsigma+2}}\\
  &\leq \frac{4\varsigma^2}{(\varsigma+2)!^2}\|h_k\|^2_{W^{\varsigma+2,\infty}(-1,1)}
  \frac{|u-1|^{\varsigma+2}|u+1|^{\varsigma+2}}{|u-1|^{2\varsigma+2}|u+1|^{2\varsigma+2}}\\
  &\leq \frac{4\varsigma^2}{(\varsigma+2)!^2}\|h_k\|^2_{W^{\varsigma+2,\infty}(-1,1)}
  \frac1{|u-1|^\varsigma|u+1|^\varsigma}\\
  &= 
  \frac{4\varsigma^2}{(\varsigma+2)!^2}\|h_k\|^2_{W^{\varsigma+2,\infty}(-1,1)}
  G_\varsigma(u)\,.
\end{align*}
It follows that 
\[
  \E\sup_{r\in[0,t]}\int_\OO G_{\varsigma,\lambda}(u(r))
  \leq K_1 + C\frac{\varsigma}{(\varsigma+2)!}C_{\H,\varsigma}
  \left(1+
  \E\int_{Q_t}G_{\varsigma,\lambda}(u)\right)
\]
and the Gronwall Lemma implies the existence of a constant $C>0$ independent 
of $\lambda$ such that 
\[
  \E\sup_{t\in[0,T]}\norm{G_{\varsigma,\lambda}(u)}_{L^1(\OO)}\leq C\,,
\]
yielding by the Fatou Lemma that 
\[
  \E\sup_{t\in[0,T]}\norm{G_{\varsigma}(u)}_{L^1(\OO)}\leq C\,.
\]
Now, this information together with analogous computations as above imply that 
\[
  \E\sum_{k=0}^\infty\int_QG'_{\varsigma}(u)^2|h_k(u)|^2 <+\infty\,,
\]
or in other words that $G_\varsigma(u)\H(u)\in L^2_\cP(\Omega; L^2(0,T; \cL^2(U,\erre)))$.
Eventually, 
the BDG inequality implies 
\begin{align*}
  &\E\sup_{r\in[0,t]}\int_0^r\left(G'_{\varsigma,\lambda}(u(s))-G_\varsigma'(u(s)), 
  \H(u(s))\,\d W(s)\right)_H\\
  &\leq C\E\left(\sum_{k=0}^\infty\int_{Q_t}
  |G'_{\varsigma,\lambda}(u)-G'_\varsigma(u)|^2|h_k(u)|^2\right)^{1/2}\,.
\end{align*}
Since $G'_{\varsigma,\lambda}(u)\to G_\varsigma'(u)$ almost everywhere,
the right-hand side goes to $0$ by the dominated convergence theorem 
since $G_\varsigma(u)\H(u)\in L^2_\cP(\Omega; L^2(0,T; \cL^2(U,\erre)))$.
Hence, we can let $\lambda\downarrow0$ in \eqref{aux_path1}
and deduce exactly \eqref{est_Gs}: this concludes the proof.
\end{proof}

\begin{lem}
  \label{lem:path2}
  There exists $K_2>0$, 
  only depending of $\varsigma$, $p$, $T$, $\OO$, $F$, $\H$, and $u_0$,
  such that 
  \beq\label{est_Vp}
  \sup_{t\in[0,T]}\norm{u(t)}_{V_p}^p
  +\int_Q|\Delta_p u|^2
  \leq K_2 + K_2\sup_{t\in[0,T]}\int_0^t
  \left(-\Delta_pu(s), \H(u(s))\,\d W(s)\right)_H
  \qquad\P\text{-a.s.}
  \eeq
  where the stochastic integral is well-defined as
  $(\Delta_pu,\H(u)\cdot)_H\in L^2_\cP(\Omega; L^2(0,T; \cL^2(U,\erre)))$.
\end{lem}
\begin{proof}[Proof of Lemma~\ref{lem:path2}]
By It\^o's formula (see \cite[Prop.~3.3]{SS-doubly2}) we have
for all $t\in[0,T]$, $\P$-a.s.~that 
\begin{align*}
&\frac{1}{p}\int_\OO |\nabla u(t)|^p + 
\int_{Q_t}|\Delta_pu|^2  + 
\int_{Q_t}F''(u) |\nabla u|^p \\
&= \frac{1}{p}\int_\OO |\nabla u_0|^p
+ (p-1)\sum_{k=0}^\infty \int_{Q_t} |h_k'(u)|^2|\nabla u|^p 
+ \int_0^t \left(-\Delta_pu(s) , \H_\lambda(u_\lambda(s)) \, \d W(s)\right)_H\,.
\end{align*}
On the left-hand side it holds by assumption {\bf B} that 
\[
  \int_{Q_t}F''(u) |\nabla u|^p \geq -C_F\int_{Q_t}|\nabla u|^p\,,
\]
whereas on the right-hand side the assumption {\bf C} yields 
\[
  (p-1)\sum_{k=0}^\infty \int_{Q_t} |h_k'(u)|^2|\nabla u|^p \leq
  (p-1)C_{\H}^2\int_{Q_t}|\nabla u^p\,.
\]
It follows for every $t\in[0,T]$ that 
\begin{align*}
\norm{\nabla u(t)}_{L^p(\OO)}^p+ 
p\int_{Q_t}|\Delta_pu|^2 
&\leq \norm{\nabla u_0}_{L^p(\OO)}^p
+ p\left[(p-1)C_\H^2 + C_F\right]\int_{Q_t}|\nabla u|^p\\
&+ p\sup_{t\in[0,T]}\int_0^t \left(-\Delta_pu(s) , \H_\lambda(u_\lambda(s)) \, \d W(s)\right)_H\,.
\end{align*}
Setting then 
\beq
  \label{K2}
  K_2:= \exp\left(p\left[(p-1)C_\H^2 + C_F\right]T\right)
  \max\left\{\norm{\nabla u_0}_{L^p(\OO)}^p, p\right\}\,,
\eeq
the thesis follows from the Gronwall lemma: this concludes the proof.
\end{proof}

\subsection{Separation of the trajectories}
Now, we are ready to prove Theorem~\ref{thm2}. 
As already outlined at the beginning of this section, the proof consists in two steps. 
First of all, we prove that there exists a set $\Omega^*\in\cF$ with $\P(\Omega^*)=1$
such that for every $\omega\in\Omega^*$ one has 
$|u(\omega,t, x)| \ne 1$ for every $(t, x) \in [0, T] \times \OO$.
Since we already know that $|u|\leq1$ almost everywhere from Theorem~\ref{thm1+},
this amounts to show that 
$|u(\omega,t,x)| < 1$ for all $(\omega,t, x) \in \Omega^*\times[0, T] \times \OO$. 

Now, since the assumption $\varsigma(p-d)>d$ guarantees in particular that 
$p>d$, by the Sobolev embedding theorem  
we have the continuous inclusion 
\[
  V_p\embed C^{0,\alpha}(\overline\OO) \qquad\forall\,\alpha\in(0,1-d/p)\,.
\]
We will denote by $c_{\alpha,p}>0$ the norm of such inclusion.
By the regularity of $u$ in Theorem~\ref{thm1+}, 
we infer that there exists $\Omega^*\in\cF$ with $\P(\Omega^*)=1$
such that, for all $\alpha\in(0,1-d/p)$,
\[
  u(\omega) \in C^0([0,T]; H)\cap L^\infty(0,T; V_p)\subset 
  C^0_w([0,T]; C^{0,\alpha}(\overline\OO)) \qquad\forall\,\omega\in\Omega^*\,.
\]
This means that every trajectory $u(\omega)$ starting from $\Omega^*$
can be evaluated {\em pointwise} in space and time, and not only almost everywhere.
Since $\Omega^*$ has full measure, it is not restrictive to assume 
that \eqref{est_Gs} and \eqref{est_Vp} hold pointwise in $\Omega^*$.

{\sc Step 1.} Let $\omega\in\Omega^*$ be fixed.
Recall that by Lemmas~\ref{lem:path1}--\ref{lem:path2} we have that 
\[
  \sup_{t\in[0,T]}\norm{G_\varsigma(u(\omega,t))}_{L^1(\OO)}
  +\sup_{t\in[0,T]}\norm{u(\omega,t)}_{V_p}^p <+\infty\,.
\]
By contradiction, let us suppose that there exists $(\bar t,\bar x)\in[0,T]\times\overline\OO$
such that $|u(\omega,\bar t,\bar x)|=1$: then, thanks to the boundary conditions
it holds that $\bar x\in\OO$ necessarily, and we have
\begin{align*}
  +\infty>\sup_{t\in[0,T]}\norm{G_\varsigma(u(\omega,t))}_{L^1(\OO)}
  &\geq \int_\OO G_\varsigma(u(\omega,\bar t, x))\,\d x=
  \int_{\OO} \frac{1}{\left|1-u(\omega,\bar{t}, x)^2\right|^\varsigma} \d x \\
  &=\int_{\OO} \frac{1}{\left|u(\omega,\bar t, \bar x)^2-u(\omega,\bar{t}, x)^2\right|^\varsigma} \d x\\
  &=\int_{\OO} \frac{1}{\left|u(\omega,\bar t, \bar x)-u(\omega,\bar{t}, x)\right|^\varsigma}
  \frac1{\left|u(\omega,\bar t, \bar x)+u(\omega,\bar{t}, x)\right|^\varsigma} \d x\,.
\end{align*}
Now, note that 
\[
\left|u(\omega, \bar{t}, \bar{x})+
u(\omega, \bar{t}, x)\right|^\varsigma
 \leq \left(1 + |u(\bar{t}, x)|\right)^\varsigma \leq 2^\varsigma \qquad \text{for a.e.~}x \in \OO\,,
\]
while the Sobolev embedding theorem guarantees that 
\begin{align*}
|u(\omega, \bar{t}, \bar{x}) - u(\omega, \bar{t}, x)|^\varsigma
&\leq \norm{u(\omega, \bar t)}^\varsigma_{C^{0, \alpha}(\overline\OO)} 
|\bar{x}-x|^{\varsigma\alpha}\\
&\leq c_{\alpha,p}^\varsigma \sup_{t\in[0,T]}\norm{u(\omega,t)}_{V_p}^\varsigma
|\bar{x}-x|^{\varsigma \alpha} \,.
\end{align*}
It follows by rearranging the terms that 
\begin{align*}
+\infty>\sup_{t\in[0,T]}\norm{G_\varsigma(u(\omega,t))}_{L^1(\OO)}\cdot
\sup_{t\in[0,T]}\norm{u(\omega,t)}_{V_p}^\varsigma
& \geq \frac{1}{c_{\alpha,p}^\varsigma2^\varsigma} 
\int_{\OO} \frac{1}{|\bar{x}-x|^{\varsigma\alpha}}\, \d x\,.
\end{align*}
However, the assumption $\varsigma(p-d)>pd$ in {\bf D}
ensures that $d/\varsigma < 1-d/p$: hence, it is possible to choose 
$\alpha\in(0,1-d/p)$ such that $\varsigma\alpha> d$. 
With this choice, the integral on the right-hand side is infinite,
but this is a contradiction since the left-had side is finite.
This shows that $|u(\omega, t,x)|<1$ for all $(t,x)\in[0,T]\times\overline\OO$, as required.

{\sc Step 2.} Let $\omega\in\Omega^*$ be fixed.
We show here that there exists $\delta = \delta(\omega)\in(0,\delta_0]$
such that $|u(\omega, t,x)|\leq 1 - \delta(\omega)$ for every 
$(t,x)\in[0,T]\times\overline\OO$. We proceed again by contradiction, 
supposing the existence of a
sequence $(t_n, x_n)_{n \in \mathbb{N}} \subset [0, T] \times \overline\OO$
for which $|u(\omega, t_n , x_n)| \to 1$ as $n\to\infty$. 
By compactness of $[0,T]\times\overline\OO$,
possibly arguing on a subsequence, we can assume that there exists 
$(\bar t, \bar x)\in[0,T]\times\overline\OO$ such that 
$(t_n, x_n)\to (\bar t, \bar x)$ as $n\to\infty$.
At this point, 
using the same notation as in {\sc Step 1}, for $\alpha\in(0,1-d/p)$ we have
\begin{align*}
|u(\omega, t_{n}, x_{n}) - u(\omega, \bar{t}, \bar{x})| 
&\leq |u(\omega, t_{n}, x_{n}) - u(\omega, t_{n}, \bar{x})| 
+ |u(\omega, t_{n}, \bar{x}) - u(\omega, \bar{t}, \bar{x})| \\
& \leq
c_{\alpha,p}\sup_{t\in[0,T]}\|u(\omega, t)\|_{V_p} |x_{n} - \bar{x}|^{\alpha} 
+ |u(\omega, t_{n}, \bar{x}) - u(\omega, \bar{t}, \bar{x})|\,.
\end{align*}
Clearly, the first term on the right-hand side converges to $0$ as $n\to\infty$.
As for the second term, we note that 
\begin{equation*}
I_{\bar{x}}(v) := v(\bar{x})\,, \quad v \in V_p\,,
\end{equation*} 
defines a continuous linear functional $I_{\bar x}\in V_p^*$ thanks to the embedding 
$V_p\embed C^0(\overline\OO)$. Consequently, 
recalling that $u(\omega)\in C^0_w([0,T]; V_p)$, we have as $n\to\infty$ that
\begin{equation*}
u(\omega, t_{n}, \bar{x}) = 
\left\langle I_{\bar{x}}, u(\omega, t_{n})\right\rangle 
\to \left\langle I_{\bar{x}}, u(\omega, \bar t)\right\rangle  = u(\omega, \bar{t}, \bar{x})\,,
\end{equation*}
so that 
\[
|u(\omega, t_{n}, \bar{x}) - u(\omega, \bar{t}, \bar{x})| \to 0\,.
\]
Putting this information together, 
we deduce that $|u(\omega,\bar t, \bar x)|=1$: however, 
this is in contrast with what we proved in {\sc Step 1}. 
This shows indeed that there exists 
$\delta = \delta(\omega)>0$
such that $|u(\omega, t,x)|\leq 1 - \delta(\omega)$ for every 
$(t,x)\in[0,T]\times\overline\OO$. The fact that $\delta(\omega)\leq\delta_0$
follows a posteriori from the initial condition.
This concludes the proof of Theorem~\ref{thm2}.


\section{Probability of separation}
\label{sec:proof3}
The aim of this section is to investigate the asymptotic behaviour 
of the probability distribution of the separation layer $\Lambda$ defined in \eqref{def:delta}
and to prove Theorem~\ref{thm3}.
Let us recall that
the random variable $\Lambda:\Omega\to(0,\delta_0]$,
whose existence is ensured by Theorem~\ref{thm2}, defines 
the separation layer of each trajectory of $u$ from the potential barriers $\pm1$.
We already know by definition of $\Lambda$ that 
its distribution is concentrated on $(0,\delta_0]$ in the sense that 
\[
  \P\{\Lambda\leq0\}=0 \qquad\text{and}\qquad\P\{\Lambda\leq\delta_0\}=1\,.
\]
Here we want to give an estimate of the probability 
\[
\P\{\Lambda\leq\delta\}=\P\left\{\sup_{(t,x)\in[0,T]\times\overline\OO}|u(t,x)|\geq1-\delta\right\}
\]
when $\delta\in(0,\delta_0)$ is close to $0$.
The main idea is to estimate the $L^\infty$-norm of $u$ in two parts, 
focusing on ``interior'' estimates and on ``boundary'' estimates, separately. 
To this end, we use the notation 
\begin{equation*}
\OO_\sigma:=\left\{x\in\OO: \inf_{y\in\partial\OO}|x-y|>\sigma\right\}\,, \qquad \OO_\sigma^c:=\overline\OO\setminus\OO_\sigma\,,\qquad\sigma>0\,.
\end{equation*}

The following lemma is crucial.
\begin{lem}
  \label{lem:prob_sep}
  For every $\alpha\in(0,1-d/p)$, there exists a constant $K>0$,
  depending only on $\alpha$, $d$, $p$, $\varsigma$, and $\OO$,
  such that 
  \beq
  \label{est_prob_sep}
  \P\{\Lambda\leq\delta\}\leq\P
  \left\{\sup_{t\in[0,T]}\norm{G_\varsigma(u(t))}_{L^1(\OO)}
  \cdot
  \sup_{t\in[0,T]}\norm{u(t)}_{V_p}^{d/\alpha}
  \geq \frac{K}{\delta^{\varsigma-\frac{d}{\alpha}}} \right\} 
  \qquad\forall\,\delta\in(0,\delta_0]\,.
  \eeq
\end{lem}
\begin{remark}
Let us point out that the estimate given by Lemma~\ref{lem:prob_sep}
is meaningful for those values $\alpha\in(0,1-d/p)$ such that $\varsigma-d/\alpha>0$.
Now, an easy computation shows that 
it is possible to choose $\alpha$ satisfying both conditions 
if and only if 
\[
  \frac{d}{\varsigma}<1-\frac{d}{p}\,,
\]
and this is ensured exactly by assumption {\bf D}.
Hence, by virtue of assumption {\bf D},
we will restrict from now on to the values 
$\alpha\in(d/\varsigma,1-d/p)$ only, so that it actually holds that $\varsigma-\frac{d}{\alpha}>0$.
\end{remark}
\begin{remark}
Note that as direct consequence of Lemma~\ref{lem:prob_sep}
and the fact that $\Lambda\leq\delta_0$ almost surely by definition, it holds that
\[
  \P
  \left\{\sup_{t\in[0,T]}\norm{G_\varsigma(u(t))}_{L^1(\OO)}
  \cdot
  \sup_{t\in[0,T]}\norm{u(t)}_{V_p}^{d/\alpha}
  \geq \frac{K}{\delta_0^{\varsigma-\frac{d}{\alpha}}} \right\} =1\,.
\]
\end{remark}
\begin{proof}[Proof of Lemma~\ref{lem:prob_sep}]
Let $\delta\in(0,\delta_0]$ and $\alpha\in(0,1-d/p)$. Recalling that 
$u\in C^0_w([0,T]; C^{0,\alpha}(\overline\OO))$,
we have the following trivial estimate
\begin{align*}
\P\{\Lambda\leq\delta\}
 &= \P\left\{\exists\, (\bar{t}, \bar{x}) \in [0,T] \times \overline\OO: \;
 |u(\bar{t}, \bar{x})|\geq 1-\delta \right\}\,.
\end{align*}
Let us estimate now the probability on the right-hand side:
we distinguish the cases when $\bar x$ is close to the boundary $\partial\OO$ or not.
To this end, let $\omega\in\Omega^*$, $\sigma>0$, and 
$(\bar t, \bar x) \in [0,T] \times \overline\OO_\sigma^c$
be such that 
\[
  |u(\omega,\bar{t}, \bar{x})|\geq 1-\delta\,.
\]
Then there exists $\bar x_\partial\in\partial\OO$ such that $|\bar x-\bar x_\partial|\leq\sigma$:
hence, thanks to the boundary conditions and the regularity of $u$, we have 
\begin{align*}
|u(\omega, \bar t, \bar{x})| = 
|u(\omega, \bar t, \bar{x}) - u(\omega, \bar t, \bar x_\partial)| 
&\leq \|u(\bar t)\|_{C^{0, \alpha}(\overline\OO)}
|\bar{x}-\bar{x}_\partial|^{\alpha} \\
&\leq c_{\alpha,p} \sup_{t\in[0,T]}\norm{u(\omega,t)}_{V_p} \sigma^{\alpha}\,.
\end{align*}
We would like to choose now $\sigma=\sigma(\omega)$
in such a way that $|u(\omega, \bar t, \bar{x})|<1-\delta$: for example, 
this can be achieved by setting 
\[
\sigma_k = \sigma_k(\omega):=
\left(\frac{\delta/k}{c_{\alpha,p} \sup_{t\in[0,T]}\norm{u(\omega,t)}_{V_p}} \right)^{\frac{1}{\alpha}}
\]
and picking $k>0$ so that $\delta<k(1-\delta)$ for all $\delta\in(0,\delta_0]$.
An elementary computation shows that this is satisfied when
\[
k>\max_{\delta\in(0,\delta_0]}\frac{\delta}{1-\delta}=\frac{\delta_0}{1-\delta_0}\,.
\]
Hence, setting $k_0:=1+\frac{\delta_0}{1-\delta_0}$ and defining
\begin{equation}\label{sigmaValue}
\sigma = \sigma(\omega):=
\left(\frac{\delta/{k_0}}{c_{\alpha,p}
 \sup_{t\in[0,T]}\norm{u(\omega,t)}_{V_p}} \right)^{\frac{1}{\alpha}}\,,
\end{equation}
from the previous considerations we obtain 
$|u(\omega,\bar t, \bar{x})| \leq \delta/{k_0}<1-\delta$. 
This shows that 
\begin{align*}
  &\left\{\omega\in\Omega^*:\;\exists\, (\bar{t}, \bar{x}) \in [0,T] \times \overline\OO: \;
 |u(\omega, \bar{t}, \bar{x})|\geq 1-\delta \right\}\\
 &=\left\{\omega\in\Omega^*:\;\exists\, (\bar{t}, \bar{x}) \in [0,T] \times \OO_{\sigma(\omega)}: \;
 |u(\omega, \bar{t}, \bar{x})|\geq 1-\delta \right\}\,.
\end{align*}
Let then $\omega\in\Omega^*$, $\bar t\in[0,T]$, and $\bar x\in\OO_{\sigma(\omega)}$
be such that $|u(\omega, \bar{t}, \bar{x})|\geq 1-\delta$.
If we denote with $B_{\sigma(\omega)}(\bar{x})$ 
the ball centred in $\bar{x}$ with radius $\sigma(\omega)$, it holds
that $B_{\sigma(\omega)}(\bar{x})\subset\OO$ since $\bar x\in\OO_{\sigma(\omega)}$,
from which 
\[
\sup_{t\in[0,T]}\norm{G_\varsigma(u(\omega,t))}_{L^1(\OO)}
\geq \int_{\OO} \frac{1}{\left|1-u(\omega, \bar{t}, x)^2\right|^\varsigma} \,\d x 
\geq \int_{B_{\sigma(\omega)}(\bar{x})} 
\frac{1}{\left|1-u(\omega, \bar{t}, x)^2\right|^\varsigma}\,\d x\,.
\]
Let us suppose with that $u(\omega, \bar{t}, \bar{x})\geq 1-\delta$
(the alternative case when $u(\omega, \bar{t}, \bar{x})\leq -1+\delta$
is analogous and the argument can be adapted easily).
For every $x\in B_{\sigma(\omega)}(\bar{x})$,
the denominator of the integrand can be bounded as follows:
\begin{align*}
\left|1-u(\omega, \bar{t}, x)^2\right|^\varsigma 
&= \left|1-u(\omega, \bar{t}, x)\right|^\varsigma\left|1+u(\omega, \bar{t}, x)\right|^\varsigma \\
& \leq 2^\varsigma \left|1-u(\omega, \bar{t}, x)\right|^\varsigma \\
&  \leq 2^\varsigma 
\left[\left|u(\omega, \bar{t}, x) - u(\omega, \bar{t}, \bar{x})\right|^\varsigma 
+ \left|u(\omega, \bar{t}, \bar{x}) - 1\right|^\varsigma\right] \\
& \leq 2^\varsigma 
\left[\norm{u(\omega, \bar t)}_{C^{0,\alpha}(\overline\OO)}^\varsigma 
|x-\bar{x}|^{\alpha \varsigma}+\delta^\varsigma\right] \\
& \leq 2^\varsigma 
\left[c_{\alpha,p}^\varsigma \sup_{t\in[0,T]}\norm{u(\omega,t)}_{V_p}^\varsigma
\sigma^{\alpha \varsigma}+\delta^\varsigma\right]\,,
\end{align*}
so that taking \eqref{sigmaValue} into account we obtain
\begin{equation*}
\left|1-u(\omega, \bar{t}, x)^2\right|^\varsigma 
\leq 2^\varsigma
\left[c_{\alpha,p}^\varsigma\sup_{t\in[0,T]}\norm{u(\omega,t)}_{V_p}^\varsigma
 \left(\frac{\delta/{k_0}}{c_{\alpha,p}\sup_{t\in[0,T]}
 \norm{u(\omega,t)}_{V_p}}\right)^\varsigma + \delta^\varsigma \right] = 
 2^{\varsigma}(k_0^{-\varsigma}+1)\delta^\varsigma\,.
\end{equation*}
Therefore, 
putting this information together 
and denoting by $a_d$ the measure of the unit ball in $\erre^d$,
we infer that 
\begin{align*}
\sup_{t\in[0,T]}\norm{G_\varsigma(u(\omega,t))}_{L^1(\OO)}&\geq
\frac1{2^{\varsigma}(k_0^{-\varsigma}+1)
\delta^{\varsigma}}\left|B_{\sigma(\omega)}(\bar x)\right|
=\frac{a_d}{2^{\varsigma}(k_0^{-\varsigma}+1)
\delta^{\varsigma}}|\sigma(\omega)|^d\\
&=\frac{a_d}{2^{\varsigma}(k_0^{-\varsigma}+1)
c_{\alpha,p}^{d/\alpha}k_0^{d/\alpha}\sup_{t\in[0,T]}\norm{u(\omega,t)}_{V_p}^{d/\alpha}}
\delta^{\frac{d}\alpha - \varsigma}
\,.
\end{align*}
Rearranging the terms, we obtain the inequality 
\[
\sup_{t\in[0,T]}\norm{G_\varsigma(u(\omega,t))}_{L^1(\OO)}\cdot
\sup_{t\in[0,T]}\norm{u(\omega,t)}_{V_p}^{d/\alpha} \geq
\frac{a_d}{2^{\varsigma}(k_0^{-\varsigma}+1)
c_{\alpha,p}^{d/\alpha}k_0^{d/\alpha}}\delta^{\frac{d}\alpha - \varsigma}\,.
\]
Hence, setting 
\beq\label{K}
  K:=\frac{a_d}{2^{\varsigma}(k_0^{-\varsigma}+1)
c_{\alpha,p}^{d/\alpha}k_0^{d/\alpha}}\,,
\eeq
we have proved that 
\begin{align*}
  &\left\{\omega\in\Omega^*:\;\exists\, (\bar{t}, \bar{x}) \in [0,T] \times \overline\OO: \;
 |u(\omega, \bar{t}, \bar{x})|\geq 1-\delta \right\}\\
 &=\left\{\omega\in\Omega^*:\;\exists\, (\bar{t}, \bar{x}) \in [0,T] \times \OO_{\sigma(\omega)}: \;
 |u(\omega, \bar{t}, \bar{x})|\geq 1-\delta \right\}\\
 &\subset\left\{\omega\in\Omega^*:\; 
 \sup_{t\in[0,T]}\norm{G_\varsigma(u(\omega,t))}_{L^1(\OO)}\cdot
\sup_{t\in[0,T]}\norm{u(\omega,t)}_{V_p}^{d/\alpha} \geq K\delta^{\frac{d}{\alpha}-\varsigma}
 \right\}\,.
\end{align*}
Since $\P(\Omega^*)=1$, this concludes the proof.
\end{proof}

We are now ready to prove Theorem~\ref{thm3}.
Recalling Lemmas~\ref{lem:path1}--\ref{lem:path2}, we know that 
\begin{align*}
  (G'_\varsigma(u), \H(u)\cdot)_H &\in L^2_\cP(\Omega; L^2(0,T; \cL^2(U,\erre)))\,,\\
  (\Delta_pu, \H(u)\cdot)_{H} &\in L^2_\cP(\Omega; L^2(0,T; \cL^2(U,\erre)))\,,
\end{align*}
so that the processes 
\begin{align*}
  M_1&:=\int_0^\cdot
  (G'_\varsigma(u(s)), \H(u(s))\,\d W(s))_H \in L^2_\cP(\Omega; C^0([0,T]))\,,\\
  M_2&:=\int_0^\cdot
  (-\Delta_pu(s), \H(u(s))\,\d W(s))_{H} \in L^2_\cP(\Omega; C^0([0,T]))\,,
\end{align*}
are well-defined square-integrable continuous real martingales that satisfy 
\begin{align*}
  \sup_{t\in[0,T]}\norm{G_\varsigma(u(t))}_{L^1(\OO)}
  &\leq K_1 + \sup_{t\in[0,T]}M_1(t) \qquad\P\text{-a.s.}\,,\\
  \sup_{t\in[0,T]}\norm{u(t)}_{V_p}^p + \norm{\Delta_p u}_{L^2(0,T; H)}^2
  &\leq K_2 + K_2\sup_{t\in[0,T]}M_2(t) \qquad\P\text{-a.s.}
\end{align*}
We will make use of the Bernstein inequality for martingales, that we recall here.
\begin{lem}(Bernstein)
\label{lem:bern}
Let $M$ be a continuous real martingale and let $l,a,b>0$.
Then
\[
  \P\left\{\sup_{t\in[0,T]}|M(t)|\geq l\,,\quad [M](T)\leq a\sup_{t\in[0,T]}|M(t)| + b\right\}
  \leq \exp\left(-\frac{l^2}{al + b}\right)\,.
\]
\end{lem}
Exploiting Berstein's inequality on $M_1$ and $M_2$, we obtain 
the following important estimates.
\begin{lem}
  \label{lem:est_M}
  There exist four constants $K_1', K_1'', K_2', K_2''>0$, only 
  depending on $\varsigma$, $p$, $T$, $\OO$, $F$,
  $\H$, and $u_0$, such that, for every $l>0$,
  \begin{align*}
  \P\left\{\sup_{t\in[0,T]}M_1 \geq l\right\} &\leq \exp\left(-\frac{l^2}{K_1'l + K_1''}\right)\,,\\
  \P\left\{\sup_{t\in[0,T]}M_2 \geq l\right\} &\leq \exp\left(-\frac{l^2}{K_2'l + K_2''}\right)\,.
  \end{align*}
\end{lem}
\begin{proof}[Proof of Lemma~\ref{lem:est_M}]
  Let us focus on $M_1$. One has 
  \[
  [M_1](T)=\int_0^T\sum_{k=0}^\infty \int_\OO |G'_\varsigma(u(s))|^2|h_k(u(s))|^2\,\d s\,,
  \]
  where, using the same computations as in the 
  proof of Lemma~\ref{lem:path1},
  \[
  |G'_{\varsigma}(u)|^2|h_k(u)|^2\leq
  \frac{4\varsigma^2}{(\varsigma+2)!^2}\|h_k\|^2_{W^{\varsigma+2,\infty}(-1,1)}
  G_\varsigma(u)\,.
  \]
  Hence, exploiting \eqref{est_Gs} we deduce that 
  \[
  [M_1](T)\leq \frac{4\varsigma^2}{(\varsigma+2)!^2} C_{\H,\varsigma}^2
  \int_0^t\norm{G_\varsigma(u(s))}_{L^1(\OO)}\,\d s
  \leq \frac{4\varsigma^2}{(\varsigma+2)!^2} C_{\H,\varsigma}^2T
  \left(K_1 + \sup_{t\in[0,T]}M_1(t)\right)\,,
  \]
  so that setting 
  \beq\label{K1'}
  K_1':=\frac{4\varsigma^2}{(\varsigma+2)!^2} C_{\H,\varsigma}^2TK_1\,,
  \qquad
  K_1'':=\frac{4\varsigma^2}{(\varsigma+2)!^2} C_{\H,\varsigma}^2T\,,
  \eeq
  we have that 
  \[
  [M_1](T)\leq K_1' + K_1''\sup_{t\in[0,T]}M_1(t) \qquad\P\text{-a.s.}
  \]
  Bernstein's inequality as in Lemma~\ref{lem:bern} yields then
  \begin{align*}
     \P\left\{\sup_{t\in[0,T]}M_1 \geq l\right\} 
     &=\P\left\{\sup_{t\in[0,T]}M_1 \geq l\,,\quad
     [M_1](T)\leq K_1' + K_1''\sup_{t\in[0,T]}M_1(t)\right\} \\
     &\leq\exp\left(-\frac{l^2}{K_1'l + K_1''}\right)\,,
  \end{align*}
  as desired. Let us turn now to $M_2$:
  exploiting \eqref{est_Vp} of lemma~\ref{lem:path2} we have 
  \begin{align*}
  [M_2](T)&=\int_0^T\sum_{k=0}^\infty \int_\OO |\Delta_pu(s)|^2|h_k(u(s))|^2\,\d s\\
  &\leq C_\H^2\int_Q|\Delta_pu|^2\leq
  C_\H^2\left(K_2 + K_2\sup_{t\in[0,T]}M_2(t)\right)\,.
  \end{align*}
  The conclusion follows then analogously from Bernstein's inequality 
  with the choices
  \beq\label{K2'}
  K_2':=K_2'':=C_\H^2 K_2\,. 
  \eeq
\end{proof}

We are now ready to conclude the proof of Theorem~\ref{thm3}.
By Lemma~\ref{lem:prob_sep}, for every $\delta\in(0,\delta_0)$
and $\alpha\in (d/\varsigma, 1-d/p)$ we have
\[
  \P\{\Lambda\leq\delta\}\leq\P
  \left\{\sup_{t\in[0,T]}\norm{G_\varsigma(u(t))}_{L^1(\OO)}
  \cdot
  \sup_{t\in[0,T]}\norm{u(t)}_{V_p}^{d/\alpha}
  \geq \frac{K}{\delta^{\varsigma-\frac{d}{\alpha}}} \right\} \,.
\]
Now, the Young's inequality implies for every $\eta\in(1,+\infty)$ that 
\[
  \sup_{t\in[0,T]}\norm{G_\varsigma(u(t))}_{L^1(\OO)}
  \cdot
  \sup_{t\in[0,T]}\norm{u(t)}_{V_p}^{d/\alpha}\leq
  \frac1\eta\sup_{t\in[0,T]}\norm{G_\varsigma(u(t))}_{L^1(\OO)}^\eta
  +\frac{\eta-1}{\eta}\sup_{t\in[0,T]}\norm{u(t)}_{V_p}^{\frac{\eta}{\eta-1}\frac{d}\alpha}\,,
\]
from which we obtain that 
\begin{align*}
  \P\{\Lambda\leq\delta\}&\leq
  \P\left\{\frac1\eta\sup_{t\in[0,T]}
  \norm{G_\varsigma(u(t))}_{L^1(\OO)}^{\eta}
  +\frac{\eta-1}{\eta}\sup_{t\in[0,T]}\norm{u(t)}_{V_p}^{\frac{\eta}{\eta-1}\frac{d}\alpha}
  \geq \frac{K}{\delta^{\varsigma-\frac{d}{\alpha}}} \right\}\\
  &\leq
  \P\left\{\sup_{t\in[0,T]}\norm{G_\varsigma(u(t))}_{L^1(\OO)}^\eta
  \geq \frac{\eta K}{2\delta^{\varsigma-\frac{d}{\alpha}}} \right\}
  +\P\left\{\sup_{t\in[0,T]}\norm{u(t)}_{V_p}^{\frac{\eta}{\eta-1}\frac{d}\alpha}
  \geq \frac{\eta K}{2(\eta-1)\delta^{\varsigma-\frac{d}{\alpha}}} \right\}\\
  &=
  \P\left\{\sup_{t\in[0,T]}\norm{G_\varsigma(u(t))}_{L^1(\OO)}
  \geq \left(\frac{\eta K}{2}\right)^{1/\eta}
  \frac1{\delta^{\frac1\eta(\varsigma-\frac{d}{\alpha})}} \right\}\\
  &\qquad+\P\left\{\sup_{t\in[0,T]}\norm{u(t)}_{V_p}^{p}
  \geq \left(\frac{\eta K}{2(\eta-1)}\right)^{\frac{\eta-1}{\eta}\frac{p\alpha}{d}}
  \frac1{\delta^{\frac{\eta-1}\eta \frac{p\alpha}{d}(\varsigma-\frac{d}\alpha)}} \right\}\,.
\end{align*}
In order to optimise the rate of convergence, we choose now $\eta\in(1,+\infty)$
so that both contributions on the right-hand side yield the same 
order, namely
\[
  \frac1\eta\left(\varsigma-\frac{d}{\alpha}\right) = 
  \frac{\eta-1}\eta \frac{p\alpha}{d}\left(\varsigma-\frac{d}\alpha\right)\,.
\]
An easy computation shows that we obtain 
\[
  \bar\eta:=1+\frac{d}{p\alpha}=\frac{p\alpha + d}{p\alpha}\,.
\]
The corresponding exponent of $\delta$ on the right-hand side is given by
\[
  \rho:=\frac1{\bar\eta}\left(\varsigma-\frac{d}{\alpha}\right)
  =p\frac{\varsigma - d/\alpha}{p + d/\alpha}>0\,,
\]
and substituting in the estimate above yields 
\begin{align*}
  \P\{\Lambda\leq\delta\}&\leq
  \P\left\{\sup_{t\in[0,T]}\norm{G_\varsigma(u(t))}_{L^1(\OO)}
  \geq\left[\frac{K(1+\frac{d}{p\alpha})}{2}\right]^{\frac{p\alpha}{p\alpha+d}}\delta^{-\rho}\right\}\\
  &+\P\left\{\sup_{t\in[0,T]}\norm{u(t)}_{V_p}^{p}
  \geq \left[\frac{K(1+\frac{p\alpha}{d})}{2}\right]^{\frac{p\alpha}{p\alpha+d}}
  \delta^{-\rho} \right\}\,.
\end{align*}
Setting now for brevity  
\beq\label{L12}
  L_1:=\left[\frac{K(1+\frac{d}{p\alpha})}{2}\right]^{\frac{p\alpha}{p\alpha+d}}>0\,, \qquad
  L_2:=\left[\frac{K(1+\frac{p\alpha}{d})}{2}\right]^{\frac{p\alpha}{p\alpha+d}}>0\,,
\eeq
Lemmas~\ref{lem:path1}--\ref{lem:path2} imply that 
\[
   \P\{\Lambda\leq\delta\}\leq
   \P\left\{\sup_{t\in[0,T]}M_1(t) \geq 
   L_1\delta^{-\rho}- K_1\right\} + 
   \P\left\{\sup_{t\in[0,T]}M_2(t) \geq 
   \frac{L_2}{K_2}\delta^{-\rho} - 1\right\}\,.
\]
We are only left to exploit the estimates of Lemma~\ref{lem:est_M}.
To this end, we shall restrict to small values of $\delta$ so that 
\[
  L_1\delta^{-\rho}- K_1> 
  \frac{L_1}2\delta^{-\rho} \qquad\text{and}\qquad
  \frac{L_2}{K_2}\delta^{-\rho} - 1>
  \frac{L_2}{2K_2}\delta^{-\rho}\,,
\]
namely
\[
  0<\delta < \min\left\{\delta_0, \frac12, \left(\frac{L_1}{2K_1}\right)^{1/\rho},
  \left(\frac{L_2}{2K_2}\right)^{1/\rho}\right\}\,.
\]
For every such $\delta$, using Lemma~\ref{lem:est_M} we infer that 
\begin{align*}
  \P\{\Lambda\leq\delta\}&\leq
  \P\left\{\sup_{t\in[0,T]}M_1(t) \geq 
  \frac{L_1}2\delta^{-\rho}\right\} + 
   \P\left\{\sup_{t\in[0,T]}M_2(t) \geq 
   \frac{L_2}{2K_2}\delta^{-\rho}\right\}\\
   &\leq
   \exp\left(
   -\frac{\frac14L_1^2\delta^{-2\rho}}{\frac12L_1K_1'\delta^{-\rho} + K_1''}
   \right)
   +
   \exp\left(
   -\frac{\frac14L_2^2K_2^{-2}
   \delta^{-2\rho}}{\frac12L_2K_2^{-1}K_2'\delta^{-\rho} + K_2''}
   \right)\,.
\end{align*}
Now, in order to get a clearer estimate, we can further restrict
the values of $\delta$ so that 
\[
   K_1''\leq \frac12L_1K_1'\delta^{-\rho} \qquad\text{and}\qquad
   K_2''\leq \frac12L_2K_2^{-1}K_2'\delta^{-\rho}\,,
\]
which yields after some easy computations, recalling that $K_1'/K_1''=K_1$
and $K_2'/K_2''=1$, 
\[
  0<\delta<\min\left\{\delta_0, \frac12,
  \left(\frac{L_1}{2K_1}\right)^{1/\rho},
  \left(\frac{L_1K_1}{2}\right)^{1/\rho},
  \left(\frac{L_2}{2K_2}\right)^{1/\rho}\right\}\,.
\]
For every such $\delta$ we obtain then
\[
  \P\{\Lambda\leq\delta\}\leq
  \exp\left(
   -\frac{L_1}{4K_1'}\delta^{-\rho}
   \right)
   +\exp\left(
   -\frac{L_2}{4K_2K_2'}\delta^{-\rho}
   \right)\,.
\]
Setting then 
\[
  L:=\frac12\min\left\{\frac{L_1}{4K_1'}, \frac{L_2}{4K_2K_2'} \right\}\,,
\]
we get exactly 
\[
\P\{\Lambda\leq\delta\}\leq 2\exp\left(-2L\delta^{-\rho}\right)\,.
\]
It is now clear that further adapting the range of $\delta$, namely 
\beq
  \label{delta_small}
  0<\delta<\delta_*:=\min\left\{\delta_0, \frac12,
  \left(\frac{L_1}{2K_1}\right)^{1/\rho},
  \left(\frac{L_1K_1}{2}\right)^{1/\rho},
  \left(\frac{L_2}{2K_2}\right)^{1/\rho}, \left(\frac{L}{\ln2}\right)^{1/\rho}\right\}\,,
\eeq
one gets $2\exp(-L\delta^{-\rho})\leq1$ for all $\delta\in(0,\delta_*)$, hence
\[
  \P\{\Lambda\leq\delta\}\leq \exp\left(-L\delta^{-\rho}\right) \qquad\forall\delta\in(0,\delta_*)\,.
\]
This concludes the proof of Theorem~\ref{thm3}.

\section{Convergence of the separation layer}
\label{sec:proof4}
In this section we prove Theorem~\ref{thm4} about the convergence of 
the random separation layers towards the deterministic one as the noise switches off.

For every $\eps\in(0,1)$ we denote by 
$u_\eps$ the unique solution of the system \eqref{1_eps}--\eqref{3_eps}
in the sense of Theorem~\ref{thm1+}, 
and by $\Lambda_\eps:\Omega\to(0,\delta_0]$ its respective 
threshold of separation from the barriers $\pm1$
in the sense of Theorem~\ref{thm2} and definition \eqref{def:delta}.
Moreover, let $\bar u$ be the unique solution to the deterministic system 
\eqref{1_eps}--\eqref{3_eps} when $\eps=0$, namely 
\begin{align*}
  \partial_t \bar u - \Delta_p \bar u + F'(\bar u) = 0 \qquad&\text{in } (0,T)\times\OO\,,\\
  \bar u=0 \qquad&\text{in } (0,T)\times\partial\OO\,,\\
  \bar u(0)=u_0 \qquad&\text{in } \OO\,.
\end{align*}
The existence and uniqueness of a strong solution $\bar u$ 
satisfying \eqref{bar1}--\eqref{bar2}
is well-known in the deterministic setting (see for example \cite{sch-seg-stef})
and can be obtained here path-by-path. Moreover, one can test the equation
by $(\bar u -(1-\delta_0))_+$ to infer that 
\begin{align*}
  \frac12\int_\OO(\bar u(t) -(1-\delta_0))_+^2
  &+\int_{Q_t\cap\{\bar u>1-\delta_0\}}|\nabla\bar u|^p
  +\int_{Q_t\cap\{\bar u>1-\delta_0\}}F'(\bar u)(\bar u -(1-\delta_0))\\
  &=\frac12\int_\OO(u_0 -(1-\delta_0))_+^2 = 0 \qquad\forall\,t\in[0,T]\,.
\end{align*}
Under the assumption that $\max\{|r_F|, |R_F|\}\leq1-\delta_0$, one has 
in particular that $F'\geq0$ on $(1-\delta_0,1)$. It follows that 
$(\bar u -(1-\delta_0))_+=0$ almost everywhere, i.e.~(exploiting the space-time
continuity of $\bar u$) that 
\[
  \bar u(\omega,t,x)\leq 1-\delta_0 \qquad\forall\,(t,x)\in[0,T]\times\overline\OO\,,
  \quad\P\text{-a.e.~}\omega\in\Omega\,.
\]
Analogously, testing by $-(\bar u+1-\delta_0)_-$, the same argument yields that 
\[
  \bar u(\omega,t,x)\geq -1+\delta_0 \qquad\forall\,(t,x)\in[0,T]\times\overline\OO\,,
  \quad\P\text{-a.e.~}\omega\in\Omega\,.
\]
Since $\|u_0\|_{L^\infty(\OO)}=1-\delta_0$, this readily implies that \eqref{bar3} holds too.

Let us focus now of the proof of the convergence of $(\Lambda_\eps)_{\eps\in(0,1)}$
to the constant deterministic threshold $\delta_0$. To this end, 
recalling the proofs of Lemmas~\ref{lem:path1}--\ref{lem:path2} and 
exploiting the fact that $\eps\in(0,1)$, it is not difficult to check that 
the constants $K_1, K_2>0$ given by \eqref{K1} and \eqref{K2}
are independent of $\eps$ and satisfy 
\begin{align}
  \label{est_Gs_eps}
  \sup_{t\in[0,T]}\int_\OO G_\varsigma(u_\eps(t))
  &\leq K_1 +  \sup_{t\in[0,T]}M_{1,\eps}(t)
  \qquad\P\text{-a.s.}\qquad\forall\,\eps\in(0,1)\,,\\
  \label{est_Vp_eps}
  \sup_{t\in[0,T]}\norm{u_\eps(t)}_{V_p}^p
  +\int_Q|\Delta_p u_\eps|^2
  &\leq K_2 + K_2 \sup_{t\in[0,T]}M_{2,\eps}(t)
  \qquad\P\text{-a.s.}\qquad\forall\,\eps\in(0,1)\,,
\end{align}
where the real-valued martingales $M_{1,\eps}$ and $M_{2,\eps}$ are (well-) defined as
\begin{align} 
  \label{M1_eps}
  M_{1,\eps}&:=\sqrt\eps
  \int_0^\cdot
  \left(G'_\varsigma(u_\eps(s)), \H(u_\eps(s))\,\d W(s)\right)_H\,, \qquad\eps\in(0,1)\,,\\
  \label{M2_eps}
  M_{2,\eps}&:=\sqrt\eps
  \int_0^\cdot
  \left(-\Delta_pu_\eps(s), \H(u_\eps(s))\,\d W(s)\right)_H\,, \qquad\eps\in(0,1)\,.
\end{align}
Analogously, it is immediate to check from the proof of Lemma~\ref{lem:prob_sep}
that for every $\alpha\in (d/\varsigma, 1-d/p)$
the constant $K$ defined in \eqref{K} is independent of $\eps$ and satisfies 
\beq
  \label{est_prob_sep_eps}
  \P\{\Lambda_\eps\leq\delta\}\leq\P
  \left\{\sup_{t\in[0,T]}\norm{G_\varsigma(u_\eps(t))}_{L^1(\OO)}
  \cdot
  \sup_{t\in[0,T]}\norm{u_\eps(t)}_{V_p}^{d/\alpha}
  \geq \frac{K}{\delta^{\varsigma-\frac{d}{\alpha}}} \right\} 
  \qquad\forall\,\delta\in(0,\delta_0]\,. 
\eeq
Similarly, going back to the proof of Lemma~\ref{lem:est_M}
we readily see that the constants $K_1', K_1'', K_2', K_2''$ defined in \eqref{K1'}
and \eqref{K2'}
are independent of $\eps$ and satisfy 
\[
[M_{i,\eps}](T)\leq K_i'\eps + K_i''\eps\sup_{t\in[0,T]}M_{i,\eps}(t) \quad\P\text{-a.s.}\,,
\qquad i=1,2\,.
\]
Now, for a technical reason that will be clear below, we define
\[
  \tilde K_1'':=\max\{K_1'', K_1'K_1+1\}>0\,, \qquad 
  \tilde K_2'':=\max\{K_2'', K_2'+1\}>0\,,
\]
so that we still have 
\[
[M_{i,\eps}](T)\leq K_i'\eps + \tilde K_i''\eps\sup_{t\in[0,T]}M_{i,\eps}(t) \quad\P\text{-a.s.}\,,
\qquad i=1,2\,.
\]
Consequently, exploiting again the Bernstein inequality
as in the proof of Lemma~\ref{lem:est_M} we obtain 
\beq\label{bern_eps}
  \P\left\{\sup_{t\in[0,T]}M_{i,\eps} \geq l\right\} \leq \exp\left(-\frac1\eps
  \frac{l^2}{K_i'l + \tilde K_i''}\right)\,,\qquad i=1,2\,.
\eeq

Now, let $\delta\in(0,\delta_0)$ and $\alpha\in (d/\varsigma, 1-d/p)$
be arbitrary. Taking into account the relations 
\eqref{est_Gs_eps}--\eqref{est_Vp_eps} and \eqref{est_prob_sep_eps}
and proceeding as at the end of Section~\ref{sec:proof3},
we obtain that 
\[
   \P\{\Lambda_\eps\leq\delta\}\leq
   \P\left\{\sup_{t\in[0,T]}M_{1,\eps}(t) \geq 
   L_1\delta^{-\rho}- K_1\right\} + 
   \P\left\{\sup_{t\in[0,T]}M_{2,\eps}(t) \geq 
   \frac{L_2}{K_2}\delta^{-\rho} - 1\right\}\,,
\]
where the constants $L_1, L_2$ are defined in \eqref{L12} and are independent of $\eps$.
Exploiting then the estimate \eqref{bern_eps} we get 
\[
  \P\{\Lambda_\eps\leq\delta\}\leq
  \exp\left(-\frac1\eps
  \frac{(L_1\delta^{-\rho}- K_1)^2}{K_1'(L_1\delta^{-\rho}- K_1) + \tilde K_1''}\right)+
  \exp\left(-\frac1\eps
  \frac{(\frac{L_2}{K_2}\delta^{-\rho} - 1)^2}{K_2'(\frac{L_2}{K_2}\delta^{-\rho} - 1) + 
  \tilde K_2''}\right)\,.
\]
Now, by the updated definition of $\tilde K_1''$ and $\tilde K_2''$ above, one has that 
\[
  K_1'(L_1\delta^{-\rho}- K_1) + \tilde K_1'' >0\,, \qquad
  K_2'\left(\frac{L_2}{K_2}\delta^{-\rho} - 1\right) + \tilde K_2'' >0\,.
\]
Hence, we can define the function $N:(0,\delta_0)\to(0,+\infty)$ as
\[
  N(\delta):=\min\left\{
  \frac{(L_1\delta^{-\rho}- K_1)^2}{K_1'(L_1\delta^{-\rho}- K_1) + \tilde K_1''},
  \frac{(\frac{L_2}{K_2}\delta^{-\rho} - 1)^2}{K_2'(\frac{L_2}{K_2}\delta^{-\rho} - 1) + 
  \tilde K_2''}\right\}\,, \qquad \delta\in(0,\delta_0)\,,
\]
and infer that 
\[
  \P\{\Lambda_\eps\leq\delta\}\leq
  \exp\left(-\frac{N(\delta)}\eps\right) \quad\forall\,\delta\in(0,\delta_0)\,.
\]

At this point, we are ready to conclude. Indeed, 
given $\eta\in(0,\delta_0)$ we have 
\[
  \P\left\{|\Lambda_\eps-\delta_0|\geq\eta\right\}=
  \P\left\{\Lambda_\eps\leq\delta_0-\eta\right\}
  \leq\exp\left(-\frac{N(\delta_0-\eta)}\eps\right)\,,
\]
from which 
\[
  \limsup_{\eps\downarrow0}\eps\ln\P\left\{|\Lambda_\eps-\delta_0|\geq\eta\right\}
  \leq -N(\delta_0-\eta)\,,
\]
and this concludes the proof of Theorem~\ref{thm4}.

\section{Examples and applications}\label{sec:ex}
In this section we propose some examples for the potential $F$ and for the diffusion operator $\mathcal{H}$, and we also highlight a possible application of the model we studied.

Let us start by giving a relevant example for the potential $F$. We recall that $F:(-1,1)\to[0,+\infty)$ is assumed to be of class $C^2$, with $F'(0)=0$, such that 
\begin{equation*}
\lim_{r\to(-1)^+}F'(r)=-\infty\,, \qquad \lim_{r\to1^-}F'(r)=+\infty\,,
\end{equation*}
and
\begin{equation*}
\exists\, C_F\geq0:\qquad F''(r)\geq - C_F \quad\forall\,r\in(-1,1)\,.
\end{equation*}
An important example of potential that satisfies the required assumptions is the so-called logarithmic potential which is defined as
\begin{equation*}
F_{\log}(r):=\frac{\theta}{2}((1+r)\ln(1+r)+(1-r)\ln(1-r))+\frac{\theta_0}{2}(1-r^2)\,, \qquad r \in (-1,1) \,,
\end{equation*}
with $0 < \theta < \theta_0$ being given constants. Such potential possesses two global minima in the interior of the physically relevant domain $[-1, 1]$, and it is the most coherent in terms of thermodynamical consistency. For these reasons, it is usually employed in contexts related to separation phenomena in physics. As already mentioned, the potential $F_{log}$ satisfies the assumptions that we required. Indeed, it holds
\begin{equation*}
F'_{\log} = \frac{\theta}{2} \log \left(\frac{1+r}{1-r}\right) - \theta_0 r\,, \qquad F''_{\log}(r) = \frac{\theta}{1-r^2}-\theta_0\,, \qquad r \in (-1,1) \,,
\end{equation*}
and therefore the hypothesis are trivially fulfilled, with $C_F = \theta_0-\theta$.

Concerning the diffusion operator $\H:\mathcal B_1\to \cL^2(U,H)$, we recall that it is defined as
\begin{equation*}
\H(v):e_k\mapsto h_k(v)\,, \quad v\in\mathcal B_1\,,\quad k\in\enne\,,
\end{equation*}
where
\begin{equation*}
(h_k)_{k\in\enne}\subset W_0^{1,\infty}(-1,1)\,, \qquad F''h_k^2\in L^\infty(-1,1) \qquad\forall\,k\in\enne \,,
\end{equation*}
and 
\begin{equation*}
C_{\mathcal H}^2:=\sum_{k=0}^\infty \left(\norm{h_k}_{W^{1,\infty}(-1,1)}^2 +\norm{F''h_k^2}_{L^\infty(-1,1)}\right) <+\infty\,.
\end{equation*}
A sequence of functions $(h_k)_{k \in \mathbb{N}}$ satisfying the assumptions that we have just mentioned, with $F = F_{\log}$, is the following
\begin{equation}\label{hk_example}
h_k(r) : = \frac{1}{k+1} (1-r^2)^2, \qquad r \in (-1,1) \,, \qquad \forall k \in \mathbb{N}\,.
\end{equation}
Indeed, there exists a constant $\mathfrak{h} > 0$ such that, for every $k \in \mathbb{N}$, it holds
\begin{equation*}
h_k (\pm 1) = h'_k(\pm 1) = 0, \qquad \|h_k\|_{L^{\infty}(-1,1)} = \frac{1}{k+1} < \infty, \qquad \|h'_k\|_{L^{\infty}(-1,1)} =  \frac{\mathfrak{h}}{k+1} < \infty \,, 
\end{equation*}
moreover, there is a constant $\mathfrak{f} > 0$ for which
\begin{equation*}
\|F''h_k^2\|_{L^{\infty}(-1,1)} = \frac{\mathfrak{f}}{(k+1)^2} < \infty \,,
\end{equation*}
and therefore
\begin{equation*}
C_{\mathcal H}^2= \sum_{k=0}^\infty \left(\frac{1}{k+1} + \frac{\mathfrak{h}}{k+1} \right)^2 + \frac{\mathfrak{f}}{(k+1)^2} < \infty \,.
\end{equation*}
However, we recall that Theorem \ref{thm2}, \ref{thm3} and \ref{thm4} require some additional conditions, i.e.\ the existence of $\varsigma\in\enne\setminus\{0\}$ such that $\varsigma(p-d)> pd$ and 
\begin{equation*}
(h_k)_{k\in\enne}\subset W_0^{\varsigma+2,\infty}(-1,1)\,, \qquad C_{\H,\varsigma}^2:=\sum_{k=0}^\infty\norm{h_k}^2_{W^{\varsigma+2,\infty}(-1,1)}<+\infty\,.
\end{equation*}
The sequence of functions defined in \eqref{hk_example} does not satisfy such assumptions, but a trivial modification of it allows to fulfil also these additional requirements. At this purpose, we consider the sequence of functions $(h_k)_{k \in \mathbb{N}}$ defined as
\begin{equation*}
h_k(r) : = \frac{1}{k+1} (1-r^2)^{\varsigma + 3}, \qquad r \in (-1,1) \,, \qquad \forall k \in \mathbb{N}\,,
\end{equation*}
which trivially satisfies the previous assumptions on the sequence $(h_k)_{k \in \mathbb{N}}$, with $F = F_{\log}$. For every $n = 0, \dots, \varsigma+2$, one can easily verify that there exists a constant $\mathfrak{h}_n > 0$ such that
\begin{equation*}
h^{(n)}_{k}(\pm 1) = 0\,, \qquad \|h_k^{(n)}\|_{L^{\infty}(-1,1)} = \frac{\mathfrak{h}_n}{k+1} <\infty \,,
\end{equation*}
and therefore
\begin{equation*}
C_{\H,\varsigma}^2 = \sum_{k = 0}^{\infty} \left(\sum_{n = 0}^{\varsigma+2} \frac{\mathfrak{h}_n}{k+1} \right)^2 < \infty \,. 
\end{equation*}

Finally, let us observe that an interesting application of our problem arises in the one-dimensional case. Indeed, under the setting of Theorems \ref{thm2}, \ref{thm3} and \ref{thm4}, we shall require that $p > d$. Therefore if $d = 1$, we can choose $p = 2$, retrieving the classical Allen-Cahn equation. We observe that, in this case, we should require $\varsigma \geq 3$, which means that the functions $(h_k)_{k \in \mathbb{N}}$ should belong at least to $W_0^{5, \infty}(-1,1)$.


\section*{Acknowledgement}
We are especially grateful to Prof.~Ulisse Stefanelli for the insightful 
discussions about separation properties for deterministic 
doubly nonlinear evolution equations.


\bibliography{ref}{}
\bibliographystyle{abbrv}

\end{document}